\documentclass[12pt,reqno]{amsart}
\usepackage[utf8]{inputenc}
\usepackage{tikz-cd}
\usepackage{mathrsfs}

\usepackage[margin=1in]{geometry}  
\usepackage{amsmath}               
\usepackage{amsfonts}              
\usepackage{amsthm}      
\usepackage{enumerate}
\usepackage{amssymb}
\usepackage{bbm}

\usepackage[utf8]{inputenc}
\usepackage{mathtools}

\usepackage{xcolor}   
\usepackage{hyperref}
\hypersetup{
    colorlinks=true, 
    linktoc=all,     
    linkcolor=black,  
}

\DeclarePairedDelimiter\abs{\lvert}{\rvert}%
\DeclarePairedDelimiter\norm{\lVert}{\rVert}%

\makeatletter
\let\oldabs\abs
\def\abs{\@ifstar{\oldabs}{\oldabs*}}
\let\oldnorm\norm
\def\norm{\@ifstar{\oldnorm}{\oldnorm*}}

\makeatother

\theoremstyle{definition}
\newtheorem{theorem}{Theorem}[section]
\newtheorem{lemma}[theorem]{Lemma}
\newtheorem{proposition}[theorem]{Proposition}

\newtheorem{corollary}[theorem]{Corollary}

\newtheorem{definition}[theorem]{Definition}
\newtheorem*{example}{Example}

\newtheorem{remark}[theorem]{Remark}

\numberwithin{equation}{section}

\DeclareMathOperator{\Aut}{Aut}

\DeclareMathOperator{\stab}{stab}

\DeclareMathOperator{\Sym}{Sym}
\DeclareMathOperator{\Cay}{Cay}

\DeclareMathOperator{\Prob}{Prob}

\let\phi\varphi

\let\empt\varnothing
 
\newcommand{\eps}{\varepsilon}

\newcommand{\acts}{\curvearrowright}
\newcommand{\actson}{\curvearrowright}

\newcommand{\EE}{\mathbb{E}} 

\newcommand{\RR}{\mathbb{R}}      

\newcommand{\ZZ}{\mathbb{Z}}      
      
\newcommand{\PP}{\mathbb{P}}

\newcommand{\NN}{\mathbb{N}}     

\newcommand{\one}{\mathbbm{1}}

\newcommand\restr[2]{{
  \left.\kern-\nulldelimiterspace 
  #1 
  \vphantom{\big|} 
  \right|_{#2} 
  }}

\newcommand{\cal}[1]{{\mathcal #1}}

\DeclareMathOperator{\MM}{\mathbb{M}}
\DeclareMathOperator{\MMo}{\mathbb{M}_0}

\newcommand{\al}{\alpha} 
\newcommand{\be}{\beta} 
\newcommand{\La}{\Lambda} 
\newcommand{\Rel}{\mathcal{R}} 
\DeclareMathOperator{\cost}{cost}
\DeclareMathOperator{\conv}{conv}
\DeclareMathOperator{\intensity}{int}
\DeclareMathOperator{\graph}{graph}
\DeclareMathOperator{\dom}{dom}
\DeclareMathOperator{\Seq}{seq}

\DeclarePairedDelimiter{\ceil}{\lceil}{\rceil}

\DeclareMathOperator{\cenum}{CluRep}
\DeclareMathOperator{\cluster}{Cl}
\DeclareMathOperator{\clusters}{Clusters}

\newcommand{\E}{\mathbb{E}}

\title{Unimodular random graphs with Property (T) have cost one}
\date{}
\author{Łukasz Grabowski, H\'ector Jard\'on-S\'anchez, Sam Mellick}
\begin{document}
\begin{abstract}
  Hutchcroft and Pete showed that countably infinite groups with Property (T) admit cost one actions, resolving a question of Gaboriau. We give a streamlined proof of their theorem, and extend it both to locally compact second countable groups and unimodular random graphs. We prove unimodular random graph analogues of the Connes--Weiss and Glasner--Weiss theorem characterising Property (T).

\end{abstract}
\maketitle

\tableofcontents

\section{Introduction}

In \cite{HP}, Hutchcroft and Pete resolved a longstanding question of Gaboriau by showing that all infinite groups with Kazhdan's Property (T) have cost one. The primary goal of this paper is to give an alternative proof of this theorem which further generalises to the setting of unimodular random graphs (URGs). This necessitates defining Property (T) for URGs and developing its basic theory, as well as proving analogues of the Connes-Weiss \cite{cw} and Glasner-Weiss \cite{GW} theorems.

Unimodular random graphs first appeared in Benjamini and Schramm \cite{BenjaminiSchrammUnimodular} as limit objects for sequences of finite graphs, and were later studied in their own right by Aldous and Lyons \cite{AldousDJ}. They are random rooted graphs which exhibit some degree of statistical homogeneity. The simplest examples of URGs are Cayley graphs of groups. URGs are intimately connected to the theory of countable Borel equivalence relations via graphings, as we will discuss below.

We prove the following characterisation of Property (T) for URGs, which we use to generalise the Hutchcroft-Pete theorem:

\begin{theorem}[See Theorem \ref{lem:completed2}]\label{def:(T)intro}
    A unimodular random graph $(G, o)$ has Property (T) if and only if for every $0 < \alpha < 1$ there exists $\kappa > 0$ such that any unimodular random 2-colouring $c$ of $(G,o)$ with $\alpha \leq \PP [\text{o is coloured red in } c(G,o)]\leq 1 - \alpha$ satisfies
    \[
    \EE [\#\{\text{neighbours of $o$ in $G$ with a different colour to $o$'s}\}] \geq \kappa. 
    \]
\end{theorem}

Cost is a fundamental invariant in measured group theory, first introduced by Levitt \cite{Levitt} and greatly developed by Gaboriau \cite{Gab}. It can be viewed as a dynamical analogue of rank, the minimum size of a generating set for a group. Cost has many applications, for instance the classification of free groups up to orbit equivalence \cite{GaboriauCout}. Further applications can be found in ergodic theory, asymptotic group theory, $3$-manifold topology, percolation theory, the study of paradoxical decompositions, and operator algebra \cite{GaboriauSeward}\cite{GaboriauLyons}\cite{AbertNikolov}\cite{RussPF}\cite{GaboriauPerc}\cite{EpsteinMonod}\cite{Tarski}\cite{OA1}\cite{OA2}.

The cost of an URG is defined as follows. An \emph{extended rewiring} of an URG $(G, o)$ is a coupling $(G, H, o)$, where $(H, o)$ is an URG with the same vertex set as $(G, o)$. Then the \emph{cost} of $(G, o)$ is
\[
   \cost(G,o) =  \inf \left\{ \frac{1}{2} \EE[\deg(H,o)] : (G,H,o) \text{ is an extended rewiring of } (G, o) \right\},
\]
where $\deg (H,o)$ denotes the degree of the root $o$ in $H$. The main result of this paper is the following.

\begin{theorem}[See Theorem \ref{thm:HP}]\label{thm:HP.intro}
    Unimodular random graphs with Property (T) have cost one.
\end{theorem} 

Cayley graphs of groups with Property (T) give natural examples of URGs with Property (T). In the appendix, a self-contained writeup of the streamlined proof of the Hutchcroft-Pete theorem is given. Further examples of URGs with Property (T) can obtained as follows.

\begin{theorem}[See Theorem \ref{thm:completed1}]\label{thm:rew.intro}
    If $(G,o)$ is an URG with Property (T) and $(G,H,o)$ is an extended rewiring, then $(H,o)$ has Property (T).
\end{theorem}

URGs with Property (T) may be constructed using point processes (see Section \ref{PPprelims} for definitions). Let $H$ be a locally compact second countable unimodular (lcscu) group and consider an essentially free invariant point process $\Pi$ on $H$. A \emph{connected factor graph} of $\Pi$ is a measurably and equivariantly defined graph $\mathcal{G}(\Pi)$ with vertex set $\Pi$ which is connected almost surely. The \emph{Palm process} $\Pi_0$ of $\Pi$ is the result of conditioning $\Pi$ to contain the identity $0 \in H$. Then $\mathcal{G}$ has a well-defined evaluation on $\Pi_0$, and $\mathcal{G}(\Pi_0)$ turns out to be an URG. For further details, see Section 3.2 of \cite{AbertMellick}.

\begin{theorem}[See Theorems \ref{TForGroupImpliesTForPalm} and \ref{otherway}]\label{thm:Palm.intro}
    An lcscu group $H$ has Property (T) if and only if $\mathcal{G}(\Pi_0)$ has Property (T) as an URG.
\end{theorem}

As an application of the above theorem to group cost, we generalise the Hutchcroft-Pete theorem:

\begin{theorem}[Theorem \ref{thm:realgs}]
    Let $G$ be a noncompact lcsc group with Property (T). Then $G$ has cost one.
\end{theorem}

Let us now give a brief outline of the paper. Property (T) for URGs is defined via \emph{measured Property (T)}. Measured Property (T) is a generalization of Property (T) introduced by Zimmer \cite{ZimT} for pmp actions and more generally by Moore \cite{MooreT} for countable Borel equivalence relations (cbers). In the main matter of the paper, we will mostly deal with cbers. The connection of URGs with cbers is the following. 

If $X$ is a standard Borel space, a cber $R$ on $X$ is a Borel subset $R\subseteq X\times X$ determining an equivalence relation with countable equivalence classes. A Borel probability measure $\mu$ on $X$ is \emph{$R$-invariant} if the Mass Transport Principle holds: for every Borel map $f\colon R\to \RR_{\geq 0}$, $$
\int_X \sum_{yRx} f(y,x) d\mu(x) = \int_X \sum_{yRx} f(x,y) d\mu(x). 
$$ 
A \emph{probability measure preserving (pmp) cber} is a triple $(X,R,\mu)$ where $X$ is a standard Borel space with a cber $R$ and $\mu$ is an $R$-invariant Borel probability measure on $X$.

A \emph{graphing} on a pmp cber $(X,R,\mu)$ is a Borel graph $\cal G\subseteq R$ such that for $\mu$-almost every $x\in X$, the $\cal G$-connected component $\cal G_x$ of $x$ is such that its vertex set $V(\cal G_x) = [x]_R$. A graphing gives rise to an URG $(\cal G_x , x)$ by sampling vertices $x$ from $\mu$. It is a well-known fact that every URG arises in this fashion \cite[Theorem 18.37]{lovasz}. For a fixed URG $(G,o)$, if $(X,R,\mu)$ is a pmp cber with a graphing $\cal G$ such that $(\cal G_x, x) = (G,o)$ in distribution, we say $\cal G$ \emph{realises} $(G,o)$. 

Measured Property (T) is defined using representations of a pmp cber, in the analogous way to how group Property (T) is defined (see Section \ref{subsec:intro(T)}). In Section \ref{sec:(T)} we prove a characterization of measured Property (T) in the style of the Connes-Weiss Theorem. An URG is defined to have Property (T) if and only if it has a realisation with measured Property (T). The well-posedness of this definition follows from the fact that measured Property (T) of a cber is completely determined by the URG obtained from any of its graphings. This is Theorem \ref{thm:completed1}. Our generalization of the Connes-Weiss Theorem for measured Property (T) reads, in the URG language, as in Theorem \ref{def:(T)intro}.

A key tool in the proof of Theorem \ref{thm:HP.intro} is a generalization of the Glasner-Weiss Theorem to URGs, which is proved in Section \ref{sec:Glasner}, and we now discuss. By a \emph{decorated unimodular random graph (DURG)}, we mean a unimodular random graph with a vertex colouring from the set $[D]:= \{1,\dots,D\}$. A \emph{DURG over an URG} $(G,o)$ is a DURG from which $(G,o)$ is obtained by forgetting the colouring. A DURG is \emph{ergodic} if every invariant event occurs with probability 0 or 1. 

\begin{theorem}[Generalized Glasner-Weiss]\label{thm:Glasner.intro}
     An URG $(G,o)$ has Property (T) if and only if the limit of any weakly convergent sequence of ergodic  DURGs over $(G,o)$ is ergodic.
\end{theorem}

The proof of the main theorem of the paper, namely Theorem \ref{thm:HP.intro} above, can be found in Section \ref{sec:HP}. The main ingredient in our proof for cost 1 is the existence of Kazhdan-optimal partitions of an URG with Property (T). We now conclude this introduction with a sketch of the proof of Theorem \ref{thm:HP.intro}. 

Let $(G,o)$ be an URG with Property (T). A unimodular $n$-colouring of $(G,o)$ is \emph{balanced} if the probability of each colour lies in the interval $(\frac{1}{n} - \frac{1}{n^2}, \frac{1}{n} + \frac{1}{n^2} )$. For each $n\in \NN$, consider $$
\kappa_n := \inf \EE [\#\{\text{neighbours of o in $G$ with a different colour}\}],
$$
where the infimum runs over all balanced unimodular $n$-colourings of $(G,o)$. As a consequence of Theorem \ref{def:(T)intro}, we have that $\kappa_n >0$. 

A \emph{Kazhdan-optimal partition} is such a balanced unimodular $n$-colouring for which the above infimum is attained. Existence of Kazhdan-optimal partitions is a consequence of the Glasner-Weiss Theorem for DURGs, Theorem \ref{thm:Glasner.intro}. 

We show that the restriction of $(G,o)$ to some colour class of a Kazhdan-optimal partition has finitely many connected components. This colour class yields a vertex percolation process on $(G,o)$ with density at most $\frac{1}{n} + \frac{1}{n^2}$ and finitely many clusters. Finally, existence of a sequence of percolation processes on $(G,o)$ with vanishing intensity and finitely many clusters implies cost one by Gaboriau's induction formula \cite[Proposition II.6]{GaboriauCout}.

The last Section of the paper, namely Section \ref{sec:pp}, is devoted to proving Theorem \ref{thm:Palm.intro} above. 

Finally, in the appendix a self-contained writeup of the streamlined proof of the original Hutchcroft-Pete theorem (for countably infinite groups) is given.

\subsection{Acknowledgments}

The authors would like to thank Anush Tserunyan and Roman Sauer for feedback that improved the readability of the paper.

{\L}G and HJS were supported by the ERC Starting Grant ``Limits of Structures
in Algebra and Combinatorics'' No.~805495.

HJS and SM were supported by the Dioscuri program initiated by the Max Planck Society, jointly managed by the National Science Centre (Poland), and mutually funded by the Polish Ministry of Science and Higher Education and the German Federal Ministry of Education and Research.

\section{Preliminaries}\label{sec:prelims}

\subsection{Cbers}

Let $X$ be a standard Borel space. A \emph{countable Borel equivalence relation (cber)} on $X$ is a Borel subset $R \subseteq X\times X$ defining an equivalence relation on $X$ with countable classes. A \emph{probability measure preserving (pmp) cber} is a triple $(X,R,\mu)$ such that $X$ is a standard Borel space with a cber $R$ and $\mu$ is a Borel probability on $X$ such that $$
\int_X \sum_{y R x} f(y,x) d\mu(x) = \int_X \sum_{y R x} f(y,x) d\mu(y),
$$
for every Borel function $f\colon R\rightarrow \RR_{\geq 0}$. A measure $\mu$ satisfying the above equality is said to be $R$-invariant. The above expression defines a measure on $R$ extending that in $X$ by letting $$
\mu (A) := \int_X  |\{y\in X : (y,x)\in A\}| d\mu(x),
$$
for every Borel subset $A \subseteq R$.

If $\Gamma \acts (X, \mu)$ is a pmp action of a countable group, then its \emph{orbit equivalence relation} $R = R(\Gamma \acts (X, \mu))$ given by
\[
    R := \{ (x, \gamma x) \in X \times X : x \in X, \gamma \in \Gamma \}
\]
is such that $(X, R, \mu)$ is a pmp cber. Moreover, it was shown by Feldman-Moore \cite{FM} that all pmp cbers arise in this way. 

If $\cal G \subseteq X\times X$ is a Borel graph on a standard Borel space, we let $\cal G_x$ denote the connected component of $x$ in $\cal G$. Such Borel graph defines a Borel equivalence relation on $X$ by letting $x \Rel (\cal G) y$ if and only if $y$ is a vertex in $\cal G_x$. A \emph{graphing} of a pmp cber is a Borel graph $\cal G \subseteq R$ such that for $\mu$-almost every $x\in X$ the vertex set of $\cal G_x$ coincides with $[x]_R$, the $R$-equivalence class of $x$. We will also say that $(X,\cal G,\mu)$ is a graphing if $(X,R_{\cal G},\mu)$ is a pmp cber.

The \emph{cost} of a pmp cber $(X,R,\mu)$ is defined as  $$
\cost (R) = \inf_{\cal G} \frac{1}{2} \int_X \deg_{\cal G}(x) d\mu(x),
$$
where $\cal G$ ranges over all graphings of $R$ and $\deg_{\cal G}$ is the degree map.

If $\Gamma$ is a countable group, then its (infimal) \emph{cost} is defined as
\[
    \cost(\Gamma) = \inf \cost(X, R, \mu),
\]
where the infimum ranges over all \emph{essentially free} pmp actions $\Gamma \acts (X, \mu)$.

A group is said to have \emph{fixed price} if all of its essentially free pmp actions have the same cost.

In Section \ref{PPprelims} we will discuss how this definition is extended to locally compact second countable unimodular groups.

If $A \subseteq X$ is a non-null Borel subset, we let $R|_A := R \cap A\times A$. Then we have that $(A, R|_A, \mu|_A)$ is a pmp cber where $\mu|_A$ is the normalization of the restriction of $\mu$ to $A$.

\begin{theorem}[Gaboriau's Induction Formula \cite{Gab}]
    Let $(X,R,\mu)$ be a pmp cber and $A\subseteq X$ a Borel subset with $\mu (A) > 0$. Then $$
    \cost (R) - 1 = \mu(A) (\cost (R|_A) -1).
    $$
\end{theorem}

Let $(X,R,\mu)$ and $(Y,Q,\nu)$ be pmp cbers.  A \emph{Borel homomorphism} from $R$ onto $Q$ is a Borel map $f\colon X\rightarrow Y$ such that $xRy$ implies that $f(x) Q f(y)$ for $\mu$-almost every $(y,x)\in R$. We say that $f$ is a \emph{factor map} if it is a Borel homomorphism such that $f_*\mu =\nu$. In this case we say that $Q$ is a \emph{factor} of $R$. A \emph{(class-bijective) extension} is a factor map $f$ such that the restriction $f|_{[x]_R} \colon [x]_R \rightarrow [f(x)]_Q$ is a bijection for $\mu$-almost every $x\in X$. We also say that $R$ is an \emph{extension} of $Q$. If $(X,\cal G,\mu)$ and $(Y,\cal H,\nu)$ are graphings, we say that $f\colon X\rightarrow Y$ is an extension if it is an extension $R_{\cal G} \rightarrow R_{\cal G}$ and for $\mu$-almost every $x\in X$ the bijection $f|_{[x]_R} \colon [x]_R \rightarrow [f(x)]_Q$ induces a graph isomorphism $\cal G_x \cong \cal H_{f(x)}$.

For any Borel subset $A\subseteq X$, its
\emph{saturation} is defined by 
$$
[A]_R:=\{x\in X:~\exists a\in A \text{ such that } aRx\}.
$$
We say that $A$ is $R$-invariant if $\mu ([A]_R\backslash A) = 0$.  A pmp cber 
$(X,R,\mu)$ is \emph{ergodic} if for any  $R$-invariant Borel subset $A\subseteq X$ we have that $\mu(A) = 0,1$.

The \emph{full group} of a pmp cber $(X,R,\mu)$ is the set $$
[R] := \{\alpha\in \Aut(X) : \alpha(x)Rx  \text{ for } \mu\text{-a.e. } x\in X\}.
$$
A sequence of Borel subsets $(A_n)_n$ of $X$ is \emph{asymptotically invariant} if for every $\phi \in [R]$ we have that $$
\mu(\phi A_n \triangle A_n) \to 0.
$$
Such a sequence is \emph{trivial} if $\mu(A_n) (1 - \mu(A_n)) \to 0$. Finally, we say that $(X,R,\mu)$ is \emph{strongly ergodic} if and only if every asymptotically invariant sequence is trivial. 

The \emph{tail equivalence relation} $E_0$ is the cber defined on $2^\NN$ by letting $x R_0 y $, for a pair $x,y \in 2^\NN$, if and only if there exists $N \in \NN$ such that $x_n= y_n$ for every $n\geq N$. A pmp cber is \emph{$E_0$-ergodic} if for every Borel homomorphism $f\colon R \rightarrow E_0$ there exists $y \in 2^\NN$ such that for $\mu$-almost every $x\in X$ we have that $f(x) E_0 y$.

\begin{theorem}[Jones-Schmidt \cite{Jones}]
An ergodic pmp cber is strongly ergodic if and only if it is $E_0$-ergodic.
\end{theorem}

Let $\cal I_R$ denote the Borel subset of Borel probabily measures on $X$ which are $R$-invariant, and $\cal E_R \subseteq \cal I_R$ the set of ergodic $R$-invariant probability measures. We recall the Ergodic Decomposition Theorem.

Each pmp cber $(X, R,\mu)$ has an \emph{ergodic decomposition}, 
that is, there exists a measure space $(Z,\nu)$ and a measurable $R$-equivariant map $X \rightarrow Z$ such that the measure disintegration $\mu = \int_Z \mu_z d\nu (z)$ satisfies that for $\nu$-a.e.~$z\in Z$, we have that $(\cal G,\mu_z)$ is ergodic.
The measure space $(Z,\nu)$ 
is called the \emph{space of ergodic components} and is essentially unique.

The following proposition encapsulates, in the language of factor maps, the fact that ergodic measures are extremal points in the space of all measures preserved by a given cber.

\begin{proposition}\label{prop-ergodic}
Let $(X,R,\mu), (Y,Q,\nu)$ be pmp cbers, suppose $Q$ is ergodic and let $\Phi\colon Y \to X$ be a factor map. Let $(Z,\rho)$ be the ergodic decomposition of $Q$. Then for almost every $z\in Z$ we have that $\Phi_\ast(\nu_z)= \mu$.  
\end{proposition}

\subsection{Unitary representations of cbers and Property (T)}
\label{subsec:intro(T)}
We now proceed to recall the concepts of measurable fields of Hilbert spaces and representation of a cber. The reader is referred to \cite[Section IV.8]{takesaki} for a thorough introduction to these concepts. 

Let $(X,\mu)$ be a standard probability space. A \emph{(measurable) field of Hilbert spaces} on $(X,\mu)$ consists of a family of Hilbert spaces $(\cal H^x)_{x\in X}$ and a subspace $\cal M$ of the vector space $\prod_{x\in X} \cal H^x$ such that 
\begin{itemize}
\item for every $\xi\in \cal M$ the function $x\in X\mapsto \|\xi(x)\|$ is $\mu$-measurable, 
\item for any $\eta \in \prod_{x\in X} \cal H_x$, if the function $x\in X \mapsto \langle \eta(x) , \xi(x)\rangle$ is $\mu$-measurable for every $\xi \in \cal M$, then $\eta \in \cal M$, and
\item there exist $\{\xi_n\}_{n\in \NN}$ such that for almost every $x\in X$, the set $\{\xi_n (x)\}_{n\in \NN}$ is total in $\cal H^x$. 
\end{itemize}

We call the elements of $\cal M$ \emph{(measurable) fields}. Whenever it is clear from the context, we will denote the measurable field of Hilbert spaces simply by $\cal H$, not making an explicit reference to the family of measurable fields. For brevity's sake, we will often refer to measurable vector fields by just fields, or to measurable fields of Hilbert spaces as fields of Hilbert spaces. Unless stated otherwise, we will consider complex Hilbert spaces. 

If $(X,R,\mu)$ is a pmp cber, a \emph{unitary (orthogonal) representation} of $R$ is a pair $(\pi,\cal H)$ where $\cal H$ is a field of complex (real) Hilbert spaces over $(X,\mu)$ together with a family of operators $\{\pi(y,x) : (y,x)\in R\}$ such that 
\begin{itemize}
\item for every $(y,x)\in R$, the operator $\pi(y,x) \colon \cal H^x \rightarrow \cal H^y$ is unitary (orthogonal), 

\item whenever $xRyRz$, we have $\pi(z,x) = \pi(z,y)\pi(y,x)$, and 

\item the map $x\in X \mapsto \langle\pi(y,x) \xi(x) , \eta(y) \rangle$  is measurable for any two measurable fields $\xi,\eta \in \cal H$.
\end{itemize}

Let $(\pi, \cal H)$ be a representation of a pmp cber $(X,R,\mu)$. A field $\xi$ in $\cal H$ is \emph{invariant} if for $\mu$-almost every $(y,x)\in R$ we have that $\pi(y,x)\xi(x) = \xi (y)$. We say that a field $\xi$ in $\cal H$ is \emph{normalized} if $\|\xi\| =1$ almost surely. The representation $(\pi,H)$ \emph{has non-trivial invariant fields} if it admits a normalised invariant field. A sequence of fields $(\xi_n)_{n\in \NN}$ is \emph{almost invariant} if for $\mu$-almost every $(y,x)\in R$, \begin{equation*}
\| \pi(y,x) \xi_n(x) - \xi_n(y) \| \rightarrow 0.
\end{equation*}
The representation $(\pi,H)$ has \emph{non-trivial almost invariant fields} if it admits an almost invariant sequence of normalised fields.

\begin{definition}
    A ergodic pmp cber has \emph{(measured) Property (T)} if and only if every representation of it with non-trivial almost invariant fields has non-trivial invariant fields. 
\end{definition}

The rest of this section is devoted to recalling some useful lemmas and fundamental facts about measured Property (T). The first one relates almost invariant fields with invariant fields.

\begin{theorem}[{\cite[Theorem 6]{pichot}}]
    Let $(R,\mu)$ be a pmp cber with Property (T) and $(\pi,\cal H)$ a representation. If $(\xi_n)_n$ is a sequence of normalized almost invariant fields then, upon possibly passing to a subsequence, there exists a sequence of normalized invariant fields $(\eta_n)_n$ such that $$
    \|\xi_n - \eta_n\| \rightarrow 0,
    $$
    almost surely.
\end{theorem}

The following is a helpful lemma for proving existence of invariant fields. The lemma is just a convenient generalization to cber representations of \cite[Proposition 1.1.5]{bhv}.

\begin{lemma}[Correlation lemma]
Let $(X,R,\mu)$ be an ergodic pmp cber and $(\pi,H)$ an orthogonal representation. Suppose there exist $\delta > 0$, a non-null Borel subset $A \subseteq X$, and a normalised field $\xi$ such that for every $(y,x) \in R|_A$, \begin{equation}\label{eq:prelim1}
\langle \pi(y,x) \xi(x) , \xi(y) \rangle \geq \delta.
\end{equation}
Then $(\pi,H)$ has non-trivial invariant fields.
\end{lemma}

\begin{proof}
Consider the field $\eta$ on $A$ defined by letting $\eta(x)$ be the unique norm minimum of the set $$
U_x := \overline{\conv}\{\pi(x,y)\xi (y) : y\in [x]_R \cap A\} \subseteq H^x,
$$ 
where $\overline{\conv}$ denotes closed convex hull. We have that the field of sets $(U_x)_x$ is invariant, in the sense that $\pi(y,x) U_x = U_y$ for every $(y,x)\in R\cap (A\times A)$. Therefore, by uniqueness of the norm minimum of closed convex sets and orthogonality of $\pi$ we must have $\eta(y) = \pi(y,x) \eta(x)$ for every $(y,x)\in R\cap (A\times A)$. For every $x\in A$ we also have that $\| \eta^x\| \geq \delta$ since Equation (\ref{eq:prelim1}) implies that for every $v \in U_x$, and in particular for $\eta(x)$, one has $\left<\xi(x), v \right>\geq \delta$. 

We can then extend $\eta$ to an invariant field on the saturation of $A$ by letting $\eta(y) = \pi(y,x)\eta(x)$, for every $y\in X\backslash A$ and some $x\in A$ measurably chosen, say, by the Lusin-Novikov Theorem. By ergodicity, the field $\eta$ is non-null almost surely and hence its normalization proves the lemma.
\end{proof}

Recall that two pmp cbers $(X,R,\mu)$ and $(Y,Q,\nu)$ are \emph{orbit equivalent} if there exists a bijective extension $f\colon X\rightarrow Y$. We say that they are \emph{weakly orbit equivalent} if there exist $A\subseteq X$ and $B \subseteq Y$ non-null subsets such that $(A,R|_A,\mu|_A)$ and $(B,R|_B,\mu|_B)$ are orbit equivalent.

\begin{theorem}[{\cite[Proposition 13]{pichot}}]
  Measured Property (T) is preserved by weak orbit equivalence.
\end{theorem}

The following result will be the main ingredient in our proof of strong ergodicity for pmp cbers with Property (T).

\begin{theorem}[{\cite[Proposition 16]{pichot}}]
Let $(X,R,\mu)$ be a pmp cber with Property (T). If $R = \cup_{n\in \NN} R_n$, where $R_1 \subseteq R_2 \subseteq\dots\subseteq R_n \subseteq \dots$ is an increasing sequence of subequivalence relations, then $R|_A=R_N|_A$ for some non-null Borel $A\subseteq X$ and $N$ sufficiently large.
\end{theorem}

To conclude this section, we point out the following corollary of non-approximability. This corollary justifies our restriction to bounded degree graphings throughout the paper. 

\begin{corollary}[{\cite[Corollary 18]{pichot}}]
Pmp cbers with Property (T) admit bounded degree graphings.
\end{corollary}

\subsection{Unimodular random graphs}

We now review the necessary background from the theory of unimodular random graphs. For a detailed introduction to the topic the reader is referred to \cite{AldousDJ}.

Throughout this section we fix some $D\in \NN$. For $d \in \NN$ we let $M^d$ be the set of all isoclasses of rooted connected locally finite graphs with degree bounded by $D$ and vertices labelled by elements of $[d]$. We will denote elements of $M^d$ specifying a representative of their class by a pair $(G,u)$, where $G$ is a coloured graph as above and $u \in V(G)$. We will denote $M:=M^1$ and think of the latter as uncoloured graphs.

We endow $M^d$ with the following compact metric:
\[
d((G,u),(H,v)) = \inf\{2^{-r} : B_{(G,u)}(r) \cong B_{(H,v)} (r)\},
\]
where $B_{(G,u)} (r)$ denotes the coloured rooted graph obtained by restricting $(G,u)$ to the ball of radius $r$ around the root $u$ and $\cong$ denotes isomorphism of rooted coloured graphs. 

We further define $M_\to^d$ to be the set of isoclasses of triples $(G,u,v)$ where $(G,u)\in M^d$ and $v$ is a neighbour of $u$ in $G$.  A \emph{decorated unimodular random graph (DURG)} is a probability measure $\Lambda$ on $M^d$ satisfying the following mass transport principle: for every $f \colon M_\to^d \to \RR_{\geq 0}$, $$
\int_{M^d} \sum_{(u,v)\in E(G)} f(G,u,v) d\Lambda (G,u) = \int_{M^d} \sum_{(u,v)\in E(G)} f(G,v,u) d\Lambda (G,u). 
$$
In the special case where $d=1$ we simply refer to $\Lambda$ as a \emph{unimodular random graph (URG)}.

We say that $\Lambda'$ is a \emph{DURG over the URG} $\Lambda$ if the URG one gets by forgetting the colours of $\Lambda'$ is $\Lambda$. In the special case where $\Lambda$ is the Cayley graph of a fixed group $\Gamma$, one can identify a DURG over $\Lambda$ with an invariant random $d$-colouring of $\Gamma$. 

Let $(\cal G,\mu)$ be a graphing and $\pi = \bigsqcup_{i=1}^d \pi_i$ a partition of its vertex space $X_{\cal G}$. There is a sampling map $\sigma \colon X_{\cal G} \to M^d$ obtained by letting $\sigma (x)$ be the $\cal G$-connected component of $x$ rooted at $x$  coloured by the corresponding indices of the partition. Aldous-Lyons \cite{AldousDJ} show the following.

\begin{proposition}\label{prop:realiza}
    Let $(\cal G,\mu)$ be a graphing and $\pi$ a partition as above. Then the pushforward $\sigma_* \mu$ is a DURG. Conversely, for every DURG $\Lambda$ there exists a graphing $(\cal G,\mu)$ together with a partition $\pi$ satisfying that $\sigma_* \mu = \Lambda$. 
\end{proposition}

A graphing $(\cal G, \mu)$ as in the above proposition is called a \emph{realization} of the DURG $\Lambda$.

The set $M^d$ is the unit space of a groupoid whose operation is defined by re-rooting. A DURG $\Lambda$ is \emph{ergodic} if and only if its only invariant events are trivial, that is, if $A\subseteq M^d$ is Borel and invariant under re-rooting then $\Lambda (A) = 0,1$. Again, Aldous-Lyons show that a DURG $\Lambda$ is ergodic if and only if it is \emph{extremal}: if $\Lambda_1$, $\Lambda_2$ are DURGs such that $\Lambda = t \Lambda_1 + (1-t)\Lambda_2$ with $0<t<1$, then $\Lambda = \Lambda_1 = \Lambda_2$.

Ergodic DURGs admit realizations by ergodic graphings. Indeed, if $\Lambda$ is an ergodic DURG and $(\cal G,\mu)$ a realization of it with ergodic decomposition $\mu = \int_Z \mu_z d\nu(z)$, then extremality ensures that $\Lambda = \sigma_* \mu = \sigma_* \mu_z$ for $\nu$-almost every $z \in Z$. 

Let $f\colon M^d \to \RR$ be a continuous function. We let $f^\to \colon M_\to^d \to \RR$ be the function defined by $$
f^\to (G,u,v) = f(G,u) - f(G,v).
$$
Similarly, we associate a finite measure on $M^d_\to$ to a dURG $\Lambda$ by letting $$
\int_{M^d_\to} g d\Lambda^\to := \int_{M^d} \sum_{(u,v)\in E(G)} g(G,u,v) d\Lambda (G,u),
$$
for every Borel $g \colon M^d_\to \to \RR$.

\begin{proposition}\label{prop:est}
    Let $f \colon M^d \to \RR$ be a continuous function and $\Lambda$ a dURG. Then $f^\to$ is continuous and $$
    \|f^\to\|_{1,\Lambda^\to} \leq 2D \|f\|_{1,\Lambda}.
    $$
\end{proposition}

\begin{proof}
    It is routine to check continuity. The norm estimate follows from unimodularity: \begin{align*}
        \|f^\to\|_{1,\Lambda^\to} &= \int_{M^d} \sum_{(u,v)\in E(G)} |f(G,v) - f(G,u)| d\Lambda (G,u)\\
        &\leq \int_{M^d} \sum_{(u,v)\in E(G)} |f(G,v)| + |f(G,u)| d\Lambda (G,u)\\
        &\leq D\|f\|_{1,\Lambda} + \int_{M^d} \sum_{(u,v)\in E(G)} |f(G,v)|  d\Lambda (G,u) \\
        &\leq 2D\|f\|_{1,\Lambda}.\qedhere     
    \end{align*}
\end{proof}

\subsection{Point processes on locally compact groups}\label{PPprelims}

We review some of the basic terminology of point processes. For a self-contained account using the same notation, see \cite{AbertMellick}. See also \cite{LastPenrose}, \cite{baccelli}, \cite{dvj1}, \cite{dvj2}. 

Let $G$ be a locally compact second countable unimodular (lcscu) group. Such a group admits a left-invariant proper metric $d$ and a left-invariant Haar measure $\lambda$. Its \emph{configuration space} is
\[
    \MM := \{ \omega \subseteq G \mid \omega \text{ is locally finite} \}.
\]
More generally, if $\Xi$ is a complete and separable metric space then we define the \emph{$\Xi$-marked configuration space} on $G$ by 
\[
    \Xi^{\MM} := \{ \omega \subseteq G \times \Xi \mid \omega \text{ is locally finite and uniquely marked} \},
\]
where $\omega \in \Xi^{\MM}$ is \emph{uniquely marked} if for all $(g, \xi) \in \omega$, if $(g, \xi') \in \omega$ then $\xi = \xi'$. In other words, we think of a $\Xi$-marked configuration as a subset of $G$ where each point has a ``mark'' or ``label'' from $\Xi$. The $\Xi$-marked configuration space is a complete and separable metric space in its own right. There is a natural action $G \acts \Xi^{\MM}$, induced by the shift action $G \acts G \times \Xi$, where the action on the second coordinate is trivial.

A \emph{point process on $G$} is a Borel probability measure $\mu$ on $\MM$. More generally, one can consider $\Xi$-marked point processes, which are Borel probability measures on $\Xi^{\MM}$. All of the definitions below apply equally well for marked point processes, in the interests of brevity we simply state them for unmarked point processes. We say $\mu$ is \emph{invariant} if $G \acts (\MM, \mu)$ is a probability measure preserving (pmp) action. The \emph{intensity} of an invariant point process $\mu$ is
\[
    \intensity(\mu) = \frac{1}{\lambda(U)}\int_{\MM} |\omega \cap U| d\mu(\omega),
\]
where $U \subseteq G$ is a subset with $\lambda(U) = 1$ (the intensity does not depend on this choice by the uniqueness of Haar measure).

A point process $\mu$ is \emph{essentially free} (or simply \emph{free}) if $\stab(\omega)$ is trivial for $\mu$ almost every $\omega \in \MM$. 

For reasons of notational coherence, we will denote the identity of $G$ by $0$, despite the fact that $G$ will rarely be an abelian group. We wish to make sense of conditioning an invariant point process $\mu$ on containing $0$. A priori this makes no sense, as this is an event of zero probability. However, one can do this rigorously by \emph{Palm theory}.

We define the \emph{rooted configuration space} as
\[
    \MM_0 = \{ \omega \in \MM \mid 0 \in \omega \},
\]
and we will sometimes refer to its elements as being \emph{rooted at $0 \in G$}. We also define the $\Xi$-marked rooted configuration space as
\[
    \Xi^{\MM_0} = \{ \omega \in \Xi^{\MM} : \exists \xi \in \Xi \text{ such that } (0, \xi) \in \omega \}.
\]
The \emph{Palm measure} associated to an invariant point process $\mu$ of finite, nonzero intensity is the measure $\mu_0$ on $\MM_0$ given by
\[
    \mu_0(A) = \frac{1}{\intensity(\mu)} \int_{\MM} \#\{g \in U \mid g^{-1}\omega \in A \} d\mu(\omega),
\]
where $U \subseteq G$ is of unit volume.

\begin{example}
The most fundamental example of an invariant point process is the Poisson point process of intensity $t$. For its definition and basic properties see \cite{LastPenrose}. It defines an essentially free pmp action of $G$. We denote its distribution by $\mathcal{P}_t$. Let $\alpha : \MM \to \MM_0$ be the ``root adding'' map given by $\alpha(\omega) = \omega \cup \{0\}$. Then the Palm measure of the Poisson point process is simply $\alpha_* \mathcal{P}$ (in fact, this characterises the Poisson point process).
\end{example}

There is a natural rerooting equivalence relation $\cal R$ on $\MM_0$, given by restricting the orbit equivalence relation of $G \acts \MM$ to $\MM_0$. Explicitly, one declares that $\omega \cal R g^{-1} \omega$ for all $g \in \omega$. 

If $\mu$ is an invariant point process, then we refer to the triple $(\MM_0, \cal R, \mu_0)$ as its \emph{Palm equivalence relation}. One can show that this is a quasi-pmp cber, and moreover pmp if the ambient group $G$ is unimodular, see Section 3.3 of \cite{AbertMellick}. Moreover, the equivalence relation is aperiodic if and only if the ambient group is noncompact.

Given a configuration $\omega$, the \emph{Voronoi tessellation} associated to it is the family of sets $\{V_g(\omega)\}_{g \in \omega}$ defined by
\[
    V_g(\omega) = \{ x \in G \mid d(g, x) \leq d(h, x) \text{ for all } h \in \omega \}.
\]
Strictly speaking, the Voronoi tessellation as defined above forms a cover of $G$ by closed sets. This can be further refined to a genuine partition into measurable sets by using a ``tie-breaking'' function. For further information, see Definition 3.15 of \cite{AbertMellick} and the discussion following it. We will work with that tessellation, and abuse notation slightly by suppressing the tie-breaking function. 

Note that the Voronoi tessellation is measurably and equivariantly defined: we have $V_{\gamma g}(\gamma \omega) = \gamma V_g(\omega)$.

Given $g \in G$, we may define an element $X_g(\omega)$ to be the unique element of $\omega$ such that $g$ is in the Voronoi cell of $X_g(\omega)$ with respect to $\omega$. Again, an equivariance property holds: $X_{\gamma g}(\gamma \omega) = \gamma X_g(\omega)$. For the purposes of working with representations, it will be convenient to introduce the following notation:
\[
    \omega_g := X_g(\omega)^{-1} \omega.
\]
In words, $\omega_g$ is the configuration that results from rooting $\omega$ at the point of $\omega$ whose cell contains $g$. Observe that the following identity is satisfied: $$
\omega_g = (g^{-1} \omega)_0.
$$ 
Thus we have defined for each $g \in G$ a map $\bullet_g : \MM \to \MM_0$, which we refer to as \emph{rooting at the cell containing $g$}.

The rooted configuration space is endowed with a countable Borel equivalence relation $\cal R$ defined by re-rooting, that is, two rooted configurations $\omega,\omega'\in \MM$ are $\cal R$ related if and only if there exists $g\in \omega$ such that $g\omega = \omega'$. 

The Palm measure of an invariant point process $\mu$ is a random configuration rooted at $0 \in G$. But we can also consider $\mu$ rooted at the cell containing the identity (that is, its pushforward under the map $\omega \mapsto \omega_0$). These two measures are different in general, but in the same measure class, as is expressed in the following formula:

\begin{theorem}[Voronoi Inversion Formula]
    Let $G$ be a lcsc group and $\mu$ a invariant point process. Then, for any Borel function $f\colon \MM \rightarrow \RR$, the following equality holds $$
    \int_{\MM} f d\mu = \intensity (\mu)\int_{\MM_0} \int_{V_0 (\omega)} f(g\omega) dh (g) d\mu_0 (\omega)
    $$
\end{theorem}

An immediate consequence of the Voronoi Inversion Formula is that the expected volume of the cell containing the identity (with respect to the Palm measure $\mu_0$) is $\intensity(\mu)^{-1}$. One can show that this implies that $\mu$ almost surely, every Voronoi cell has finite volume.

The \emph{cost} of an invariant point process $\mu$ is defined by
\[
    \cost(\mu) - 1 := \intensity(\mu) (\cost(\MM_0, \cal R, \mu_0) - 1).
\]
It is shown in \cite{AbertMellick} that cost as defined above is an isomorphism invariant. Note that the above definition can be interpreted as an infinitesimal version of the induction formula of Gaboriau. 

The \emph{cost} of a nondiscrete lcscu group $G$ is then defined as
\[
    \cost(G) := \inf \cost(\mu),
\]
where $\mu$ varies over the essentially free point processes on $G$.

    It was known in the community for some time that there was a way to extend the definition of cost to lcscu groups via the means of ``cross-sections'', which we will soon define. This method first appears in implicitly in \cite[Proposition 4.3]{KPV}, where it is described as part of folklore, and then explicitly in \cite{Carderi}. This method is essentially equivalent to the point process version given above, as described in \cite[Section 3.5]{AbertMellick}. The point process perspective and techniques have proved to be useful in proving fixed price, as in \cite{AbertMellick}, \cite{FMW}, and \cite{Mellick2023}.

Let $G \acts (X, \mu)$ be a pmp action of an lcsc group. A Borel set $Y \subseteq X$ is a \emph{cross-section} for the action if there exists a neighbourhood of the identity $U \subseteq G$ such that the map $(u, y) \mapsto uy$ is injective and $\mu(X \setminus GY) = 0$. Cross sections always exist -- for history and further information, see Section 4.2 of \cite{kechriscber}.

Given an essentially free pmp action $G \acts(X, \mu)$, a choice of cross-section $Y \subseteq X$ gives the ``orbit viewing map'' $V : X \to \MM$, given by
\[
    V(x) = \{ g \in G : g^{-1}x \in Y \}.
\]
One can check that this map is equivariant, and thus the pushforward $V_* \mu$ is a point process. The uniform separation set $U \subseteq G$ for the cross-section guarantees that this process has finite intensity. Cross-sections have the structure of a quasi-pmp cber (pmp when the group is unimodular), and under the map $V$ this can be identified with the Palm equivalence relation of $V_* \mu$. 

It is known that that class-bijective extensions of a cross-section equivalence relation for an action $G \acts (X, \mu)$ arise as the cross-section equivalence relation of some extension $G \acts (\tilde{X}, \tilde{\mu})$ of the action. For a proof, see Proposition 8.3 of \cite{BowenHoffIoana}.

\subsection{Gaussian Hilbert spaces}

In this subsection we will briefly recall some basic facts about Gaussian Hilbert spaces. The reader interested on the theory of Gaussian Hilbert spaces and the proofs of the facts here mentioned is referred to \cite{janson} and \cite[Appendix A.7]{bhv}.

Let $(\Omega,\nu)$ be a standard probability space. A Borel function $Z\colon \Omega \to \RR$ is a \emph{gaussian} if $Z_* \nu$ is a gaussian measure on $\RR$, that is, there exist $\mu\in \RR$ and $\sigma\geq 0$ such that $$
\nu (Z^{-1} A) = \frac{1}{\sqrt{2\pi\sigma} }\int_A \exp\left[- \frac{1}{2} \left(\frac{x - \mu}{\sigma}\right)^2 \right] dx,
$$
for every Borel subset $A\subseteq \RR$. 

We say that $Z$ is \emph{centered} if $\EE [Z] = \mu = 0$. If $Z$ is a centered gaussian we have that $\|Z\|_2 = \sigma$. An identity along the lines of the following lemma is at the core of the proof of the Connes-Weiss Theorem \cite{cw}, which we will generalize to cbers in Theorem \ref{thm:CW}.

\begin{lemma}\label{lem:gauss}
    Let $(\Omega,\nu)$ be a standard probability space and $X,Y$ normalized centered gaussians in $L^2(\Omega,\nu)$. Then $$
    \nu(\{X \geq 0\}\cap \{Y<0\}) = \frac{1}{\pi} \cos^{-1} \EE[XY]
    $$
\end{lemma}

A \emph{Gaussian Hilbert space (GHS)} is a Hilbert subspace $K$ of $L_\RR^2(\Omega,\nu)$  where every random variable $Z\in K$ is a centered Gaussian. Every real separable Hilbert space is isomorphic to a GHS in which the Gaussians generate the $\sigma$-algebra of the underlying space. When the Gaussians of the GHS $K$ generate the $\sigma$-algebra of the underlying space, every orthogonal operator $T:K\rightarrow K$ is induced by an essentially unique measure-preserving bijection $\theta_T$ of $\Omega$ in the sense that $T=\theta_T^*$, where $\theta_T^* Z := Z\circ \theta_T^{-1}$ for every $Z\in L_\RR^2 (\Omega,\nu)$. Due to uniqueness, the equality $\theta_{TS} = \theta_T \theta_S$ holds almost surely for any pair $S,T$ of orthogonal operators.

\section{Measured Property (T) and URGs}\label{sec:(T)}

The first aim of this section is to prove a generalization of the Connes-Weiss Theorem in the setting of ergodic pmp cbers. In order to prove this characterization, we will need to show invariance of measured Property (T) under extensions. We will deduce from the Connes-Weiss style characterization that measured Property (T) is invariant under factor maps, generalizing the fundamental fact that group Property (T) is preserved quotients. These results show that measured Property (T) is completely determined by the URG obtained by sampling from any bounded degree graphing. We conclude this section with the introduction of Property (T) for URGs.

\begin{proposition}\label{prop:passup}
Ergodic extensions of pmp cbers with Property (T) have Property (T).
\end{proposition}
\begin{proof}
Let $(X,R,\mu)$ and $(Y,Q, \nu)$ be ergodic cbers on and $h\colon Y \rightarrow X$ an extension. Assume $(R,\mu)$ has Property (T) and fix a representation $(\pi,\cal H)$ of $(Q,\nu)$. We begin by constructing a representation $(\bar\pi, \overline{ \cal H})$ of $R$. 

The measurable field of Hilbert spaces $\overline{ \cal H}$ is defined by\begin{equation*}
x\in X \longmapsto \overline{ \cal H}^x=\int_{h^{-1} \{x\}}^\oplus \cal H^y d\nu_x (y),
\end{equation*}
where $(\nu_x)_{x\in X}$ is the disintegration of $\nu$ with respect to $h$. Given $x\in X$ and $y\in Y$ such that $h(y) R x$, we let $x_y$ be the unique element of $Y$ such that $x_y Q y $ and $h(x_y)=x$. This is defined for almost every $x\in X$. We then define the unitaries of the representation by\begin{equation*}
\bar\pi(x',x)=\int_{h^{-1} \{x\}}^\oplus \pi(x'_y,x_y)d\nu_{x} (y),
\end{equation*}
for every $(x',x)\in R$.

If $(\xi_n)_n$ is an almost invariant sequence of normalised fields for $(\pi,\cal H)$, then one may consider the normalized fields in $\overline{\cal H}$ given by \begin{equation*}
x\in X \longmapsto \bar\xi_n^x := \int_{h^{-1} \{x\}}^\oplus \xi_n (y) d\nu_x (y)
\end{equation*}
for each $n\in \NN$. Observe that for almost every $(x',x)\in R$, $$
\|\pi(x',x)\bar\xi_n (x) - \bar\xi_n (x')\|^2 = \int_{h^{-1}\{x'\}} \|\pi(x'_y,x_y)\xi_n(x_y) - \xi_n ({x'_y})\| ^2 d\nu_{x'} (y),
$$
so by the Dominated Convergence Theorem and almost invariance of the $(\xi_n)_n$, it follows that $(\bar\xi_n)_n$ is an almost invariant sequence of normalized vectors for $(\bar\pi,\overline{\cal H})$. Therefore, by Property (T), there exists an invariant normalised field $\bar\eta$ for $\overline{\cal H}$. 

Consider the field $y\in Y \mapsto\eta (y) := \bar\eta (h(y))(y)$. Defined as such, $\eta$ is normalised and, for almost every $(y',y)\in Q$, \begin{align*}
\pi(y',y) \eta(y) &= \pi(y',y) [\bar\eta(h(y)) (y)]\\
&=  [\bar\pi(h(y'),h(y))\bar \eta (h(y))](y') \\
&= \bar\eta (h(y'))(y')\\
&= \eta (y'),
\end{align*}
so it is also invariant. Therefore $(Q,\nu)$ has Property (T) by definition.
\end{proof}

Our next aim is to proof a characterization of measured Property (T) \`a la Connes-Weiss. We begin showing that measured Property (T) implies strong ergodicity.

\begin{proposition}\label{prop:ergodiplust}
Pmp cbers with measured Property (T) are strongly ergodic.
\end{proposition}
\begin{proof}

 Let $(X,R,\mu)$ be a pmp cber with Property (T) which is not strongly ergodic. Then by the Jones-Schmidt Theorem, there exists a Borel homomorphism $f \colon X \to 2^\NN$ such that there exists no $y \in E_0$ for which $f(x) E_0 y$ for almost every $x\in X$. 

Let $R_n = f^{-1} (F_n)\cap R$, where $(F_n)_n$ is an increasing sequence of finite Borel subequivalence relations of $E_0$ such that $\bigcup_n F_n = E_0$. For $f$ being a Borel homomorphism, we have that $R = \bigcup_n R_n$, so by non-approximability of Property (T) cbers there exists $A \subseteq X$ non-null such that $R|_A = R_N|_A$ for some $N$ sufficiently large. 

Let $s$ be a Borel selector for $F_N$ and consider the map $h:=s\circ f|_A$. Due to our assumptions on $f$, we have that the measure $h_*\mu|_A$, where $\mu|_A$ is the normalization of the restriction of $\mu$ to $A$, is not supported on a single atom. Therefore, if $B$ is a Borel subset of $2^\NN$ such that $0<h_*\mu|_A (B) <1$, we have that $h^{-1}(B)$ is an $R|_A$-invariant set with $0 < \mu (h^{-1}(B)) < \mu (A)$. This implies that $R$ is not ergodic, contradicting our assumptions.
\end{proof}

Propositions \ref{prop:passup} and \ref{prop:ergodiplust} imply Schmidt's theorem \cite{schmidt} in our setting: every ergodic extension of a pmp cber with Property (T) is strongly ergodic. In the next theorem we adapt the argument of Connes-Weiss \cite{cw} to prove the converse.

\begin{theorem}[Connes-Weiss for cbers]\label{thm:CW}
An ergodic pmp cber has Property (T) if and only if every ergodic extension of it is strongly ergodic.
\end{theorem}

\begin{proof}
It is only left to show sufficiency. Let $(X,R,\mu)$ be an ergodic pmp cber without Property (T). We will show that $(X,R,\mu)$ admits an ergodic extension which is not strongly ergodic. Since $(X,R,\mu)$ does not have Property (T), after considering the realification of an appropriate unitary representation, there exists an orthogonal representation $(\pi,\cal K)$ of $(X,R,\mu)$ with an almost invariant sequence of normalised fields $(Z_n)_n$ but no invariant normalised fields. 

By ergodicity, we can assume that $\cal K^x = K$ for a fixed GHS $K$ on a standard probability space $(\Omega,\nu)$ in which the Gaussians generate the $\sigma$-algebra. For every $(y,x) \in R$ we then have that $\pi (y,x) = \theta_{y,x}^*$ for a measure-preserving automorphism $\theta_{y,x}$ of $\Omega$. Moreover, we have that $\theta_{z,y}\theta_{y,x} = \theta_{z,x}$ almost surely. 

We consider the extension $Q$ of $R$ defined on the standard probability space $(X\times \Omega, \mu\otimes\nu)$ by $$
(x,\omega) Q (y,\omega') \iff xRy \text{ and } \theta_{y,x}^{-1} \omega = \omega'.
$$
In order to ease notation, let us denote $Z_n (x)(\omega) := Z_n (x,\omega)$. We will show that the Borel subsets $$
    A_n := \{(x,\omega)\in X\times \Omega : Z_n(x,\omega) \geq 0 \}
$$
form an asymptotically invariant sequence for $Q$. We will then conclude the proof by passing to an ergodic component where they are still asymptotically invariant and non-trivial.

For each $\alpha \in [R]$, we let $\theta_{\alpha,x} := \theta_{\alpha x,x}$ and then define a full group automorphism $\tilde\alpha \in [Q]$ by $$
\tilde\alpha (x,\omega) := (\alpha x, \theta_{\alpha, x}^{-1} \omega).
$$

We may then compute that for any such $\alpha\in [R]$, \begin{align*}
\notag (\mu\otimes \nu) (A_n \triangle  \tilde\alpha A_n) &= \int_{X} \nu\left(\{Z_n (x,\cdot) \geq 0 \}\triangle \{Z_n (\alpha^{-1} x, \theta_{\alpha^{-1},x} \cdot )  \geq 0\} \right) d\mu(x) \label{eq:CW1}\\
&= \int_{X} \nu\left(\{Z_n(x) \geq 0 \}\triangle \{\pi(x,\alpha^{-1}x)Z_n(\alpha^{-1} x) \geq 0\} \right) d\mu(x)\\
&\notag =\int_{X} \nu\left(\{Z_n(x) \geq 0\} \cap \{(\pi(x,\alpha^{-1} x) Z_n(\alpha^{-1} x)   < 0\} \right) d\mu(x) \\
&  +\int_{X} \nu\left(\{Z_n(x) < 0\} \cap \{(\pi(x,\alpha^{-1} x) Z_n(\alpha^{-1} x)   \geq 0\} \right) d\mu(x) \\
&= \frac{2}{\pi} \int_X \cos^{-1}\EE_\nu [Z_n (x) (\pi(x, \alpha^{-1} x) Z_n(\alpha^{-1}x))] d\mu(x)\\
&= \frac{2}{\pi} \int_X \cos^{-1}\EE_\nu [Z_n (\alpha x) (\pi(\alpha x,  x) Z_n(x))] d\mu(x)
\end{align*}

By almost invariance of the $(Z_n)_n$ we have that $\EE_\nu [Z_n (\alpha x) (\pi(\alpha x,  x) Z_n(x))] \rightarrow 1$ for almost every $x\in X$. Hence, for the $Z_n$ being normalized fields we deduce from the Dominated Convergence Theorem that $(\mu\otimes \nu) (A_n \triangle \tilde\alpha A_n) \rightarrow 0$ for every $\alpha \in [R]$. For $Q$ being an extension of $R$, we have that $[x]_Q = \{\tilde\alpha x: \alpha \in [R]\}$, and so $(\mu\otimes \nu) (A_n \triangle \alpha A_n) \rightarrow 0$ for every $\alpha \in [Q]$.

Let $(B,\beta, ((\mu\otimes\nu)_b)_{b\in B})$ be the ergodic decomposition of $(X\times \Omega , Q,\mu\otimes \nu)$. We show that there exists a positive measure subset of ergodic components for which $(\mu\otimes\nu)_b (A_n)$ stays bounded away from 0. To do this, fix $n\in \NN$ and consider the Borel set $$
\Phi := \left\{\{y,x\}\in [X]^2 : yRx \text{ and } \EE_\nu [Z_n (y) (\pi(y,  x) Z_n(x))]\leq \frac{1}{2}\right\}.
$$
By \cite[Lemma 7.3]{KM}, there exists a $\Phi$-maximal finite partial subequivalence relation $F$. We may regard $F$ as a partial matching with $\graph (\alpha_n) \subseteq R$. 

We now show that actually $\alpha_n \in [R]$. For this we only need to show that $\mu (\dom (\alpha_n)) = 1$. Suppose otherwise that $C:= X\backslash \dom (\alpha_n) $ were not $\mu$-null. Then, by $\Phi$-maximality of $F$, for any $(y,x)\in R|_{C}$ we would have that $$
\EE_\nu [Z_n (y) (\pi(y,  x) Z_n(x))]> \frac{1}{2}.
$$
But then it  would follow from the correlation lemma that $(\pi,\cal K)$ had normalized invariant vectors, a contradiction. 

We have thus constructed for every $n\in \NN$ a full group element $\alpha_n \subseteq [R]$ such that $$
\EE_\nu [Z_n (y) (\pi(y,  x) Z_n(x))]\leq \frac{1}{2},
$$
for almost every $(y,x)\in R$. But then, $$
(\mu\otimes \nu)(A_n \triangle \tilde\alpha_n A_n) \geq \frac{2}{\pi}\cos^{-1}\frac{1}{2} \geq \frac{2}{3}.
$$
We deduce that $(\mu\otimes\nu)_b (A_n \triangle \tilde\alpha_n A_n) $ stays bounded away from zero on a positive measure set of ergodic components, concluding the proof, as there exists then $b\in B$ where $(A_n)_n$ is a non-trivial asymptotically invariant sequence for $(X\times \Omega, Q, (\mu\otimes\nu)_b)$. 
\end{proof}

The Connes-Weiss theorem for cbers can be used to show invariance of Property (T) under factor maps.

\begin{proposition}\label{prop:passdown}
Factors of pmp cbers with Property (T) have Property (T).
\end{proposition}

\begin{proof}
Let $(X, R,\mu)$ be a pmp cber with Property (T) and $(Y, Q,\nu)$ an ergodic factor determined by $\phi\colon X\rightarrow Y$. Consider any ergodic extension $(Z,P,\rho)$ of $(Y,Q,\nu)$ given by $\psi \colon Z\rightarrow X$. Our goal is to show that $(Z,P,\rho)$ is strongly ergodic, as then Property (T) for $(Y,Q,\nu)$ follows from the Connes-Weiss Theorem for cbers.

We will show that there exists an ergodic extension $(W,S,\lambda)$ of $R$ which factors onto $P$ by a factor map $\delta \colon W \rightarrow Z$. This is enough to conclude the proof. Indeed, for $R$ having Property (T) we have that $S$ is strongly ergodic. Then if $f\colon Z\rightarrow 2^\NN$ is a Borel homomorphism $P\rightarrow E_0$, we have that its composition with $\delta$ is a Borel homomorphism $S\rightarrow E_0$. Hence, there exists $a\in 2^\NN$ such that $f(\delta(w))E_0 a$ for almost every $w\in W$. But then $f(z)E_0 a$ for almost every $z\in Z$. It follows that $P$ is $E_0$-ergodic, so strongly ergodic by the Jones-Schmidt Theorem.

Existence of $(W,S,\lambda)$ follows from an adaptation of the joining construction for ergodic systems \cite[Chapter 6]{joinings} to our setting. We sketch this construction for the convenience of the reader. Let $W$ be the following pushout $$
W:= X\ast_Y Z = \{(x,z) \in X\times Z : \phi (x) = \psi (z)\}.
$$
The equivalence relation $S$ is then defined as a skew-product. For each $(x,z)\in W$ and $x' Rx$ we define $(x',x) z$ as the unique element of $Z$ satisfying $$
\psi((x',x) z) = \phi (x') \text{ and } (x',x) z P z
$$
Then we let $(x,z)S(x',z')$ if and only if $x'Rx$ and $z' = (x',x)z$. The extension and factor maps on $R$ and $P$ are given by projection on first and second components respectively. Letting $(\mu_y)_{y\in Y}$ and $(\rho_y)_{y\in Y}$ be the decompositions of $\mu$ and $\rho$ respectively over $(Y,\nu)$, we define $$
\lambda := \int_Y \mu_y \otimes \rho_y d\nu (y).
$$ 
This measure is invariant, though not necessarilly ergodic. To conclude the proof it suffices to pass to an ergodic component by Proposition \ref{prop-ergodic}
\end{proof}

Invariance under extensions and factor maps allows us to prove stability under finite index subequivalence relations. For an ergodic pmp cber $R$ and a Borel subequivalence relation $Q\leq R$, recall that $[Q:R]\in \NN \cup \{0,\infty\}$ is the number of $Q$-classes in almost every $R$-class. Following \cite[Section 25]{KM}, we say that $Q$ has finite index in $R$ if $[Q:R]< \infty$.

\begin{proposition}\label{prop:fini}
    Let $(X,R,\mu)$ be an ergodic pmp cber and $Q \leq R$ an ergodic finite index subequivalence relation. Then $R$ has Property (T) if and only if $Q$ has Property (T).
\end{proposition}

\begin{proof}
It suffices to show that there exists an ergodic extension $(Y,\tilde R,\tilde\mu)$ of $(X,R,\mu)$ which is weakly orbit equivalent to $(X,Q,\mu)$. Then Property (T) of $(X,R,\mu)$ will be equivalent to that of $(Y,\tilde R,\tilde\mu)$, by invariance under extensions and factors, which will also be equivalent to  Property (T) of $(X,Q,\mu)$, by invariance of Property (T) under weak orbit equivalence.
    
    Let $n:=[Q:R]<\infty$. Observe first that there exists a finite equivalence relation $F \leq R$ such that $|[x]_F| = n$ for almost every $x \in X$ and for any two $xFy$ we have that $x$ is not $Q$-related to $y$. To produce such $F$, let $$
    \Phi := \{\phi \subseteq [X]^{n} : \phi\times\phi \subseteq R, \forall x,y\in \phi\text{ such that } x\neq y,   x\not\sim_Q y \}.
    $$
    By \cite[Lemma 7.3]{KM}, there exists a $\Phi$-maximal finite partial subequivalence relation $F$. This implies that for every $x\in \dom (F)$, we have that $[x]_F \in \Phi$. Moreover, by ergodicity of $Q$ we have that $X\backslash \dom (F)$ contains less than $n$ vertices in almost every $R$-class, so for $R$ being pmp it follows that $\mu (\dom (F))= 1$.

    We define $Y:= X\times [n]$ together with the Borel equivalence relations $(x,i) \tilde Q (y,j)$ if and only if $x Q y$ and $i = j$, and $(x,i)\tilde R (y,j)$ if and only if $x R y$. When $Y$ is endowed with $\tilde\mu$, the product measure of $\mu$ and the normalized counting measure on $[n]$, we have that both $(X\times [n],\tilde Q,\tilde\mu)$ and $(X,\tilde R,\tilde\mu)$ are pmp. 

    Using a Borel linear order on $X$, we may construct an essentially free action $\ZZ/n\ZZ \actson X$ such that $F$ is the orbit equivalence relation of the action. Using this action we construct a Borel map $f \colon X\times [n] \rightarrow X$ defined by $f(x,i):= (i-1)\cdot x$. We then have that, by construction of $F$, the map $f$ witnesses that $\tilde R$ is an extension of $R$. Therefore $(X\times [n], \tilde R,\tilde\mu)$ has Property (T). For Property (T) being weakly orbit equivalent, it follows that $(X\times \{1\}, \tilde R|_{X\times \{1\}}, \tilde \mu|_{X\times \{1\}} )$ has Property (T), so  $(X,Q,\mu)$ has Property (T), since $\tilde R|_{X\times\{1\}} = \tilde Q|_{X\times\{1\}}$.
\end{proof}

Invariance under extensions and factors implies that Property (T) of a pmp cber is completely determined by the URG generated by any of its graphings, as it is made precise in the following corollary of Theorem \ref{thm:CW}. This observation motivates the introduction of Property (T) for URGs below.  

\begin{corollary}\label{cor:TandBi}
Let $(X,\cal G,\mu)$ and $(Y,\cal H, \nu)$ be ergodic realizations of an URG $\Lambda$. Then $\Rel (\cal G)$ has Property (T) if and only if $\Rel (\cal H)$ has Property (T).
\end{corollary}
\begin{proof}
In the language of \cite{lovasz}, the graphings $\cal G$ and $\cal H$ are locally equivalent by hypothesis, so they are bi-locally isomorphic by \cite[Theorem 18.59]{lovasz}, that is, there exists a third realization $(Z,\cal I,\rho)$ of $\Lambda$ such that $\Rel (Z)$ is a common extension of $\Rel (\cal G)$ and $\Rel (\cal H)$. By proposition \ref{prop-ergodic}, we can assume  that $(\cal I , \rho)$ is ergodic after passing to an ergodic component. The conclusion follows from invariance of Property (T) under extension and factor maps.
\end{proof}

\begin{definition}\label{def:Tforurg}
    An URG $\Lambda$ has Property (T) if and only if it admits an ergodic realization $(X,\cal G,\mu)$ such that $(X,\Rel (\cal G),\mu)$ has Property (T).
\end{definition}

For the convenience of the reader, we compile other characterizations of Property (T) for URGs in terms of their realizations in the following theorem. Further characterizations will be given in the next section, in Theorems \ref{lem:completed2} and \ref{thm:glasner}. 

\begin{theorem}\label{thm:completed1} Let $\La$ be a URG. The following conditions are equivalent.
\begin{enumerate}
\item $\La$ has Property (T).
\item Every ergodic realization of $\Lambda$ has Property (T).
\item Every ergodic realization of $\Lambda$ is strongly ergodic.
\end{enumerate}
\end{theorem}

\begin{proof}
 $(a)\implies (b)$ follows from  Corollary \ref{cor:TandBi}. Now $(b)$ implies that every ergodic realization of $\Lambda$ has Property (T), so by Proposition \ref{prop:ergodiplust}, every ergodic realization is strongly ergodic, proving $(c)$. 
 
 To conclude the proof, let us assume $(c)$ and let $(X,\cal G,\mu)$ be an ergodic realization of $\Lambda$. If $(Y,Q,\nu)$ is an ergodic extension of $(X,\Rel (\cal G),\mu)$ by a Borel map $f$, then $Q$ admits a graphing $\cal H$ defined by $(y,y')\in \cal H$ if and only if $(f(y),f(y'))\in \cal G$ and $yQy'$. For $f$ being an extension, we have that $(Y,\cal H, \nu)$ is an ergodic realization of $\Lambda$, so it is strongly ergodic. Therefore, every ergodic extension of $(X,\Rel (\cal G),\mu)$ is strongly ergodic and by the Connes-Weiss Theorem for cbers, it has Property (T). Hence $\Lambda$ has Property (T) by definition.
\end{proof}

\section{A Glasner-Weiss Theorem for URGs}\label{sec:Glasner}

The aim of this section is to prove a Glasner-Weiss Theorem for URGs (see Theorem \ref{thm:glasner}). This Theorem will be crucial in our proof of existence of cost 1 extensions of pmp cbers with measured Property (T). Precisely, it will be the main ingredient in the proof of existence of Kazhdan-optimal partitions (see Theorem \ref{thm-optimal-existence}).

Our next goal is to show that realisations of a given URG with Property (T) are strongly ergodic in a uniform fashion, in the following sense: 

\begin{theorem}\label{lem:completed2}
Let $\Lambda$ be an URG. Then $\Lambda$ has Property (T) if and only if for every $\eps \in (0,1)$ there exists $\kappa=\kappa_\eps >0$ such that for every ergodic realization $(\cal G,\mu)$ of $ \Lambda$ and every Borel $A \subseteq X_{\cal G}$ with $\eps \leq \mu(A) \leq 1-\eps$ we have $\mu (\partial_{\cal G} A) \geq \kappa$.
\end{theorem}

Before proving this theorem, we recall the inverse limit construction for graphings.

An \emph{inverse system of graphings} $(\cal G_n,\mu_n,\phi_n)_{n}$ consists of a sequence of graphings $(\cal G_n,\mu_n)_{n}$ and a sequence of extensions $\phi_n:\cal G_n \rightarrow \cal G_{n-1}$ defined for each $n\in \NN$. Let
\[
X_{\cal G_\infty} = \left\{ x \in \prod_{n\in \NN} X_{\cal G_n} : \phi_n (x_n) = x_{n-1},\ \text{for every}\ n\in \NN \right\},
\]
and set \begin{equation*}
\cal G_\infty = \left\{(x,y)\in X_{\cal G_\infty} \times X_{\cal G_\infty} : (x_n,y_n) \in \cal G_n,\ \text{for every}\ n\in \NN \right\}.
\end{equation*}
Kolmogorov's consistency theorem implies that there exists a unique Borel probability measure $\mu_\infty$ on $X_{\cal G_\infty}$ making $(\cal G_\infty,\mu_\infty)$ and the Borel \emph{projection} maps $\pi_n : x\in X_{\cal G_\infty} \rightarrow x_n \in X_{\cal G_n}$ an extension for every $n\in \NN$. The graphing $(\cal G_\infty, \mu_\infty)$ is the \emph{inverse limit} of the inverse system $(\cal G_n,\mu_n,\phi_n)_{n\in \NN}$.

\begin{proposition} \label{prop:ergolim}
The inverse limit of an inverse system of ergodic graphings is ergodic.
\end{proposition}
\begin{proof}
Let $(\cal G_n , \mu_n, \phi_n)_{n}$ be an inverse system of ergodic graphings. By Proposition \ref{prop-ergodic}, almost every ergodic component of $\mu_\infty$ makes the maps $\pi_n$ extensions for every $n\in \NN$. It follows from the uniqueness clause in the Kolmogorov consistency theorem that almost every ergodic component of $\mu_\infty$ coincides with $\mu_\infty$, so $(\cal G_\infty,\mu_\infty)$ is ergodic.
\end{proof}

\begin{proof}[Proof of Theorem \ref{lem:completed2}]
Sufficiency follows directly from the definition of URG with Property (T). To prove necessity, let $\Lambda $ be a URG and suppose there exists $\eps \in (0,1)$ and a sequence of ergodic realizations $(\cal G_n,\mu_n)_{n}$ of $\Lambda$ with Borel subsets $A_n \subseteq X_{\cal G_n}$ such that $\eps \leq \mu (A_n) \leq 1-\eps$ and $\mu (\partial_{\cal G_n} A_n) \rightarrow 0$. We will produce an ergodic realization of $\Lambda$ which is not strongly ergodic from this data.

Let $(\cal H_1 , \nu_1) := (\cal G_1, \nu_1)$ and define inductively $(\cal H_n , \nu_n)$ to be a common ergodic extension the graphings $(\cal H_{n-1} , \nu_{n-1})$ and $(\cal G_n , \mu_n)$. Existence of $(\cal H_n , \nu_n)$ follows again from \cite[Theorem 18.59]{lovasz} after passing to an ergodic component, if necessary, by Proposition \ref{prop-ergodic}. Let $\phi_n : \cal H_n \rightarrow \cal H_{n-1}$ and $\alpha_n : \cal H_n \rightarrow \cal G_n$ be the corresponding extensions, and let $(\cal H_\infty , \nu_\infty)$ be the inverse limit of the inverse system of graphings $(\cal H_n,\nu_n,\phi_n)_{n}$, with $n$-th projection map denoted by $\pi_n$. We get the following diagram for the system.

\begin{center}
\begin{tikzcd}

  & \cal H_2 \arrow[dl,"\phi_1"]\arrow{d}{\alpha_1} & \cal H_3 \arrow[l,"\phi_3"]\arrow{d}{\alpha_3} & \dots\arrow{l}{\phi_4} &  \cal H_n \arrow{l}{\phi_n}\arrow{d}{\alpha_n} &\dots\arrow{l}{\phi_{n+1}} & \cal H_{\infty} \arrow{l} \\
\cal G_1 &  \cal G_2 & \cal G_3  & \dots   & \cal G_n & \dots &
\end{tikzcd}
\end{center}

For each $n\in \NN$, let $B_n = \pi_n^{-1} \alpha_n^{-1} (A_n) \subseteq X_{\cal H_\infty}$. Since $\pi_n$ and $\alpha_n$ are extensions, we get $\eps \leq \nu_\infty (B_n) \leq 1-\eps$ for every $n\in \NN$ and $\nu_\infty ( \partial_{\cal H_\infty} B_n) = \nu_n (\partial_{\cal H_n} A_n) \rightarrow 0$. Hence $\cal H_\infty$ is not strongly ergodic but ergodic by Proposition \ref{prop:ergolim}. 
\end{proof}

The following is our Glasner-Weiss Theorem \cite{GW} for unimodular random graphs.

\begin{theorem}\label{thm:glasner}
    An URG $\Lambda$ has Property (T) if and only if the limit of any weakly convergent sequence of ergodic  DURGs over $\Lambda$ is ergodic.
\end{theorem}

\begin{proof}

If $\Lambda$ does not have Property (T), then by definition it admits an ergodic realization $(\cal G,\mu)$ which is not strongly ergodic. Hence, there exists a sequence of almost invariant sets $(A_n)_n$ of vertices such that $\eps \leq \mu (A_n)\leq 1 -\eps$ for some $\eps>0$ and every $n\in \NN$, and $\mu (\partial_{\cal G} A_n) \to 0$. Let $\Lambda_n$ be the sequence of ergodic DURGs associated to the partition $A_n \sqcup A_n^c$ as in Proposition \ref{prop:realiza}. Upon passing to a subsequence we may assume that the $\Lambda_n$ weakly converge to some $\Lambda_\infty$. 

 Let \begin{equation}\label{eq:jag1}
K_d := \{(G,u,v)\in M^d_\to : u\ \text{and}\ v\ \text{have a different colour}\}.
\end{equation}
Since $\mu (\partial_{\cal G} A_n) \rightarrow 0$, we have that $\Lambda_n^\to (K_2) \rightarrow 0$. As $K_2$ is clopen, it follows from weak convergence $\Lambda^\infty (K_2) = 0$, so $\Lambda_\infty$-almost every graph has all vertices with the same colour. However, since $\eps \leq \mu (A_n)\leq 1 -\eps$, both colours appear at the root with positive probability. Hence $\Lambda_\infty$ is not ergodic. 

Conversely, let $(\Lambda_n)_n$ be a weakly convergent sequence of ergodic DURGs over $\Lambda$ and $\Lambda_\infty$ their limit. Suppose that $\Lambda_\infty$ is not ergodic. We will show that $\Lambda$ does not have Property (T) by negating the conclusion of Theorem \ref{lem:completed2}. Then, there exists a Borel re-rooting invariant set $A \subseteq M^d$ with $0<\Lambda_\infty (A) <1$. Let $(A_k)_k$ be a sequence of clopen subsets of $M^d$ such that $\Lambda_\infty (A \triangle A_k) \to 0$. 

We have that $\one_{A_k}\to \one_A$ in mean, so by Proposition \ref{prop:est} we get that $\one_{A_k}^\to \to \one_A^\to$ in mean. Observe that $\one_A^\to$ is null $\Lambda_\infty$-almost surely, and that $\EE_{\Lambda_\infty}[\one_{A_k}^\to] = \Lambda_\infty^\to (\partial_E A_k)$, where $$
\partial_E A_k := \{(G,u,v) \in M_\to^d : (G,u)\in A_k, (G,v)\not\in A_k\}.
$$
The latter is a clopen subset of $M_\to^d$ corresponding to the edge boundary of $A_k$. It then follows that $\Lambda_\infty^\to (\partial_E A_k) \rightarrow 0$. 

For each $k$ there exists $n(k)$ sufficiently large such that $|\Lambda_{n(k)} (A_k) - \Lambda_\infty (A_k)|<\frac{1}{k}$ and $|\Lambda_{n(k)}^\to (\partial_E A_k) - \Lambda_{\infty}^\to (\partial_E A_k)|<\frac{1}{k}$ hold. Hence, we have that $\Lambda_{n(k)}^\to (\partial_E A_k) \rightarrow 0$ and $\Lambda_{n(k)}(A_k) \rightarrow \Lambda_\infty (A)$. Considering ergodic realizations of the $\Lambda_{n(k)}$, we deduce that $\Lambda$ does not have Property (T) by Theorem \ref{lem:completed2}.
\end{proof}

\section{Hutchcroft-Pete Theorem for cbers}\label{sec:HP}

In this section we prove our generalization of the Hutchcroft-Pete Theorem \cite{HP}.

\begin{theorem}\label{thm:HP}
Infinite URGs with Property (T) have cost one. Equivalently, pmp cbers with Property (T) have a cost 1 extension.
\end{theorem}

The proof of this theorem relies on the existence of a ergodic realizations of $\Lambda$ with arbitrarily small sets with finitely many clusters. Let $(\cal G,\mu)$ be a graphing and let $A\subseteq X$. An \emph{$A$-cluster} is any maximal subset $S$ of $A$ with the property that $A$ is path-connected in $\cal G$. In other words, an $A$-cluster is the set of vertices in a connected component of the graph induced by $A$ in $\cal G$. We say that  $(\cal G,\mu)$ has \emph{finitely many $A$-clusters} if for $\mu$-almost every $x$ we have that $x^{\cal G}$ contains only finitely many $A$-clusters.

Precisely, Theorem \ref{thm:HP} follows from the following.

\begin{theorem}\label{thm-hp-urg}
Let $\La$ be an infinite URG with Property (T) and $\eps>0$. Then there exists  an ergodic realization $(\cal G,\mu)$ of $\Lambda$ and a Borel subset $A\subseteq X$ with $\mu(A)<\eps$ such that $(\cal G,\mu)$ has finitely many $A$-clusters.
\end{theorem}

The rest of this section is devoted to the proof of this theorem. Let us start with a very informal sketch of the proof. We consider all possible ergodic realizations $(\cal G,\mu)$ of  $\La$. In each of them we consider partitions of the vertex space $X_{\cal G}$ into $\ceil{\frac{1}{\eps}}$ roughly equal parts. In Theorem \ref{thm-optimal-existence}, we show that there exists a graphing together with a partition which minimises the amount of edges which go between the different parts of the partition. We call such minimising partitions \emph{Kazhdan-optimal}. The remaining ingredient for the proof of Theorem~\ref{thm-hp-urg} is showing that in a Kazhdan-optimal partition, some Borel subset of its parts has finitely many clusters. This is done in Section \ref{sec:HP}.c. 

Before we dive into the proof of Theorem \ref{thm-hp-urg}, let us show how it implies Theorem \ref{thm:HP}. Roughly, the proof goes as follows. 

\begin{proof}[Proof of Theorem \ref{thm:HP}]
Let $D$ be the maximum degree of $\Lambda$. For any $\eps > 0$, by Theorem \ref{thm-hp-urg}, there exists an ergodic realization $(\cal G,\mu)$ of $\Lambda$ and $A\subseteq X$ with $\mu(A) < \eps$ such that $(\cal G,\mu)$ has finitely many $A$-clusters. We show that $\cost (R_{\cal G},\mu) \leq 1 + \eps (D-1)$ using Gaboriau's Induction Formula. By the latter, it suffices to show that $\cost (R_{\cal G}|_A , \mu|_A) \leq D$ where $\mu|_A$ is the normalization of the restriction of $\mu$ to $A$.

Let $\cal E \subseteq A \times A$ denote the Borel subset of pairs $(x,y)$ such that $x$ and $y$ are $R_{\cal G}$-related, belong to diferent $A$-clusters, and realize the $\cal G$-distance between their clusters. Roughly speaking, adding an edge to $\cal G|_A$ for each of the pairs in $\cal E$ yields a graphing whose connected components are classes of $R_{\cal G}|_A$. By iteratively removing the edges defined by $\cal E$, but without losing the connectivity, we obtain graphings for $R_{\cal G}|_A$ whose expected degree converges to that of $\cal G|_A$, thus proving the above statement.

Precisely, we define a countable Borel equivalence relation $\cal Q$ on $\cal E$ by letting $(x,y)\cal Q (x',y')$ if and only if $x$ and $x'$ belong to the same $A$-cluster and $y$ and $y'$ belong to the same $A$-cluster as well. The equivalence relation $R_{\cal G}$ is not smooth almost surely for being pmp. Therefore $\cal Q$ is aperiodic a.s. We let $(\cal E_n)_n$ be a vanishing sequence of markers \cite[Lemma 6.7]{KM} for $\cal Q$. We then have that $R_{\cal G|_{A} \cup \cal E_n}$ is connected for every $n$ and has finite expected degree. Moreover, we get that $$
\cost (R_{\cal G}|_A , \mu|_A) \leq \liminf_n \frac{1}{2} \EE_{\mu|_A}[\deg (\cal G|_{A} \cup \cal E_n)] \leq D,
$$
and the proof is complete.
\end{proof}

\subsection{Existence of Kazhdan-optimal partitions}

For $k\in \NN$, let $\Prob(k)$ be the set of all sequences $\al = (\al_1,\al_2,\ldots,\al_{k})$ of non-negative real numbers with $\sum_{i=1}^{k} \al_i = 1$. For  $\al,\be \in \Prob(k)$ we let 
$$
    d_{\infty} (\al,\be) := \max\{|\al_i - \be_i|\colon i\in \{1,\ldots, k\}\},
$$
and we let $\Seq := \bigcup_{k=1}^\infty \Prob(k)$.

A \emph{partition of length $k$} of $X$ is a finite sequence of $k$ disjoint Borel subsets of positive measure  whose union is $X$. If $\pi = (\pi_1,\pi_2,\ldots, \pi_{k})$ is a partition and $\mu$ is a probability measure on $X$ then we let $W_\mu(\pi)$ be the sequence
$(\mu(\pi_1),\mu(\pi_2),\ldots,\mu(\pi_{k}))$. We will write $W(\pi)$ instead of $W_\mu(\pi)$ when $\mu$ is clear from the context. For $U,V\subseteq X_{\cal G}$ we define 
$$
    \partial_{\cal G} (U,V):=\{(x,y)\in \cal G \colon  x\in U,\, y\in V\}.
$$  
The \emph{Kazhdan constant $K_{\cal G,\mu}(\pi)$ of $\pi$} is defined to be the sum
$$
\sum_{1\leq i <j\leq k}\mu( \partial_{\cal G} (\pi_i,\pi_j)).
$$
Thus $K_{\cal G,\mu}(\pi)$ measures the amount of edges of $(\cal G,\mu)$ which go between the parts of $\pi$.

If $(\cal G,\mu)$ is a graphing then we define the \emph{Kazhdan function of $(\cal G,\mu)$}
$$
    K_{\cal G,\mu}\colon \Seq \times [0,1]\to [0,\infty]
$$ 
as follows. For $ \al = (\al_1,\al_2,\ldots,\al_k)\in\Seq$  and $\eps\in [0,1]$ we let 
$$
    K_{\cal G,\mu}(\al,\eps) :=\inf\{K_{\cal G,\mu}(\pi)\colon \pi \text{ is a partition with }d_{\infty}(W(\pi),\al)\leq \eps\}.
$$

If $\La$ is an ergodic URG then we define the Kazhdan function $K_\La\colon \Seq\times[0,1]\to [0,\infty]$ as follows: 
for $ \al = (\al_1,\al_2,\ldots,\al_k)\in\Seq$  and $\eps\in [0,1]$ we let 
$$
K_\La(\al,\eps) = \inf\{K_{\cal G,\mu}(\al,\eps)\colon \cal G\text{ is an ergodic realization of }  \La\}.
$$

Suppose we are given a URG $\La$, an ergodic realization $(\cal G,\mu)$, a partition $\pi$ of $X_{\cal G}$, and  
$(\al,\eps) \in \Seq\times [0,1]$. We say that $\pi$ is 
\emph{Kazhdan-optimal} for $(\al, \eps)$ if $K_{\cal G,\mu}(\pi) = K_\La(\al,\eps)$ and 
$d_\infty(W(\pi),\al) \leq \eps$.

The following theorem shows the existence of ergodic realizations with Kazhdan optimal partitions.

\begin{theorem}\label{thm-optimal-existence}
Let $\La$ be an ergodic URG with Property (T).
\begin{enumerate}
\item For every $(\al,\eps)\in \Seq\times[0,\min(\al))$ we have that $K_\La(\al,\eps)>0$.
\item For every $(\al,\eps) \in \Seq\times [0,\min(\al))$ there exists an ergodic graphing $\cal G$ and a partition $\pi$ of $X_{\cal G}$ which is Kazhdan-optimal for $(\al,\eps)$.
\end{enumerate}
\end{theorem}

\begin{proof}
If we had $K_\La(\al,\eps)=0$, there would exist a sequence $(\cal G_n , \mu_n)$ of ergodic realizations of $\Lambda$ together with partitions $\pi_n = \bigsqcup_{i=1}^d \pi_n^i$ such that $d(W(\pi_n),\alpha) \leq \eps$ and $K_{(\cal G_n,\mu_n)}(\pi_n)\rightarrow 0$. Then it follows that $\mu_n(\partial_{\cal G_n} \pi_n^1) \rightarrow 0$ whilst $0< \alpha_1 - \eps<\mu_n(\pi_n^1)<\alpha_1 + \eps < 1$. Theorem \ref{lem:completed2} would then imply that $\Lambda$ does not have Property (T), proving part (a).

In order to prove part (b), we consider a sequence $(\cal G_n , \mu_n)$ of ergodic realizations of $\Lambda$ together with partitions $\pi_n = \bigsqcup_{i=1}^d \pi_n^i$ such that $d(W(\pi_n),\alpha) \leq \eps$ and $K_{(\cal G_n,\mu_n)}(\pi_n)\rightarrow K_\Lambda (\alpha,\eps)$. Let $\Lambda_n$ be the sequence of ergodic DURGs associated to these graphings and their partitions. Since $M^d$ is compact, upon possibly passing to a subsequence, we may assume that there exists a DURG $\Lambda_\infty$ over $\Lambda $ such that $\Lambda_n \to \Lambda_\infty$ weakly. It follows from Theorem \ref{thm:glasner} that $\Lambda_\infty$ is ergodic.

Let $A_1 ,\dots,A_d$ be the Borel partition of $M^d$ defined by the colour of the root. Each of these sets is clopen, so as $\Lambda_\infty$ is a weak limit of the $\Lambda_n$ we have that $|\Lambda_\infty (A_i) - \alpha_i|\leq \eps$. Moreover, if we let now $$
C := \{(G,u)\in M^d :u\text{ has a neighbour with a different colour}\},
$$
then $C$ is clopen, thus satisfying that $$
K_{(G_n,\mu_n)}(\pi_n) = \Lambda_n (C) \to \Lambda_\infty (C).
$$
By ergodicity of $\Lambda_\infty$ it admits an ergodic realization $(\cal G,\mu)$. Combining all of the above, we deduce that the partition of the vertices of $\cal G$ induced by the sampling map and the colouring in $M^d$ is Kazhdan-optimal.
\end{proof}

\subsection{Clusterwise Bernoulli extensions}\label{clust.ext}

Let $(X,\cal G, \mu)$ be a bounded degree graphing and let $A\subseteq X$ be a Borel set. In this subsection we provide details for the construction of an extension used in the proof of Theorem~\ref{thm-hp-urg}. In essence, an element of this extension is a pair $(x, f)$, where $x \in X$ and $f$ is a colouring of the $A$-clusters in the equivalence class of $x$ produced by independently flipping coins.

For $x\in A$  we define $\cluster(x)\subseteq [x]_{\cal G}$ to be the $A$-cluster containing $x$, and for $x\in X$ we let $\clusters(x)$ be the set of all $A$-clusters contained in $[x]_{\cal G}$.  To simplify the notation we will assume that for every $x\in X$  we have that $\clusters(x)$ is infinite.

The following lemma is an exercise in Borel combinatorics. It states that clusters can be measurably, though not canonically, enumerated. 

\begin{lemma}
There exists a Borel map
$$
    \cenum\colon X \to A^\NN,
$$ 
such that for every $x\in X$ we have that 
$$
    \cluster(\cenum_0(x)), \cluster(\cenum_1(x)),\ldots
$$ 
is an enumeration of the set $\clusters(x)$. 
\end{lemma} 

The clusterwise Bernoulli extension is defined by means of the following cocycle. The \emph{cluster cocycle} $c\colon R \to \Sym(\NN)$ is defined as follows. For every $(x,y)\in R$ we let 
$c(x,y)$ be the unique bijection of $\NN$ such that for all 
$i\in \NN$ we have that $\cluster(\cenum_i(x)) = \cluster(\cenum_{c(y,x)(i)}(y))$. One can check that $c$ satisfies the cocycle relation that $c(x,z) = c(x,y)c(y,z)$.

Let us fix $\eps\in (0,1)$. We define the \emph{clusterwise Bernoulli extension} $(Y,\cal H,\nu)$ as follows. First, we set $Y = X \times \{0,1\}^\NN$, and we let $\nu$ be the product of $\mu$ and the Bernoulli measure on $\{0,1\}^\NN$ which assigns the measure $\eps$ to $\{1\}$. Finally we put an edge between $(x,f)$ and $(y,g)$ if $(x,y)\in \cal G$ and $f = g\circ c_{x,y}$. This means that we connect $(x,f)$ and $(y,g)$  if $x$ and $y$ are connected in $\cal G$ and $f$ and $g$ describe the same function on $\clusters(x)=\clusters(y)$. 

The fact that $\cal H$ preserves $\nu$ follows from the fact that bijections of $\NN$ preserve the Bernoulli measure on $\{0,1\}^\NN$ (this is an example of a skew product).  The extension map $\Phi\colon Y \to X$ is defined on vertices by $\Phi(x,f) = x$. This finishes the construction of $(Y,\cal H,\nu)$.

\subsection{Clusters in Kazhdan-optimal partitions}

We now prove Theorem~\ref{thm-hp-urg}. Let us first discuss the proof informally. The goal of the theorem is to show that the parts of a Kazhdan-optimal partition have finitely many clusters. If some of these parts had infinitely many clusters then we could merge a very small amount of these clusters to a neighbouring part. This is made precise by means of a clusterwise Bernoulli extension. By means of this process, we would obtain a new partition with a strictly smaller Kazhdan constant, contradicting Kazhdan-optimality.

\begin{proof}[Proof of Theorem~\ref{thm-hp-urg}]
Let us fix $\La$ an URG with Property (T), a natural number $n \in\NN$, a probability sequence $\al\in \Prob(n)$, and $\eps\in(0,\min(\al))$. After Theorem \ref{thm-optimal-existence}, there exists an ergodic realization $(\cal G,\mu)$ together with a Kazhdan-optimal partition $\pi$. Our aim is to show that for some $i\in [n]$ there exists a non-null Borel $U\subseteq \pi_i$ such that $(\cal G,\mu)$ has finitely many $U$-clusters, for every $i\in [n]$. This is a slightly more general conclusion than that of Theorem \ref{thm-hp-urg}. To recover the conclusion of the latter Theorem one may just set $\alpha = (1/n, \dots, 1/n)$ and $\eps = 1/n^2$.

After permuting the elements of $\al$ and of $\pi$ we may assume that for some $l\in \{1,\ldots, n-1\}$ we have that $\mu(\pi_j) \leq \al_j$ for $j<l$  and $\mu(\pi_j) \geq \al_j$ for $j \geq l$.
 
Consider $A: = \bigcup_{j< l} \pi_j$ and $B:= \bigcup_{j\geq l} \pi_j$. Since $(\cal G,\mu)$ is ergodic, we have that $\mu(\partial_{\cal G}(A,B)) >0$. It follows that for some $a < l$ and $b\geq l$ we have 
\begin{equation}\label{ij}
    \mu(\partial_{\cal G}(\pi_a, \pi_b)) >0.
\end{equation}
Let us very informally describe the strategy now. We want to construct a new partition $\bar 
\pi$ by  ``moving some small amount of $\pi_b$-clusters into 
$\pi_a$''. As such the new partition $\bar \pi$ will miss at least some 
of the edges of $\partial_{\cal G}(\pi_a,\pi_b)$, and hence we will 
have $K(\bar \pi) < K(\pi)$.  Furthermore, the ``small amount'' will be 
specified in such a way that $d_\infty( W(\bar \pi),\al)\leq \eps$, 
which will contradict the Kazhdan-optimality of $\pi$.

Let us now make it precise. First, for $x\in \pi_b$ let $\cluster(x)\subseteq x^{\cal G}$ be the $\pi_b$-cluster of $x$, and for $x\in X_{\cal G}$ let $\clusters(x)$ be the set of all $\pi_b$-clusters contained in $x^{\cal G}$. Let $U\subseteq \pi_b$ be the union of those $\pi_b$-clusters which are connected by an edge in $\cal G$ to an element of $\pi_a$. By~\eqref{ij} we have that $U$ has positive measure. 

If $U$ has finitely many clusters then the proof is finished. Thus let us assume that $U$ does 
not have finitely many clusters. Since $(\cal G,\mu)$ is ergodic, we deduce that for $\mu$-almost every $x\in X_{\cal G}$ we have that $x^{\cal G}$ contains infinitely many $U$-clusters. 

Let $\Phi\colon (\cal H,\nu)\to (\cal G,\mu)$ be the 
$U$-clusterwise Bernoulli extension. Recall that an  element of $X_{\cal H}$ is a pair $(x,f)$ where $x\in X_{\cal G}$ and 
$f\colon \clusters(x) \to \{0,1\}$. 

Let $\bar \pi_i := \Phi^{-1}(\pi_i)$ and  $(Z,\eta, (\nu_z)_{z\in Z})$ an ergodic decomposition of $(\cal H,\nu)$. By Proposition~\ref{prop-ergodic}, we have that for $\eta$-almost every $z\in Z$ we have that $(\cal H,\nu_z) \to (\cal G,\mu)$ is an extension. It follows that  for $\eta$-almost every $z\in Z$ the graphing $(\cal H,\nu_z)$ is an ergodic realization of $\Lambda$ and that $K_{\cal H,\nu_z} (\bar \pi) = K_{\cal G,\mu}(\pi)$.

For $i\in \{0,1\}$, let 
$$
    \bar U_i := \{(x,f)\in X_{\cal H}\colon x\in U, f(\cluster(x)) =i\}
$$ 

Note that  $\eps\geq \nu(\bar U_1)=\eps\cdot \mu(U)>0$. Furthermore, the set 
$$
    Y:=\{ (x,f)\in X_{\cal H}\colon \exists C,D\in \clusters(x)\colon f(C)=0, f(D)=1\}
$$
has $\nu$-measure $1$, since $x^{\cal G}$ contains infinitely many $U$-clusters for $\mu$-almost every $x\in X_{\cal G}$.   Since for every $y\in Y$ we have that $y^{\cal H}$ intersects both $\bar U_0$ and $\bar U_1$, we deduce that for almost every $z\in Z$ we have $\nu_z(\bar U_0) =0 \iff \nu_z(\bar U_1)=0$. By Proposition~\ref{prop-ergodic} we have  $\nu_z(\bar U_0 \cup \bar U_1)= \mu(U) >0$ for almost every $z\in Z$, and so we deduce in particular that for almost every $z\in Z$ we have $\nu_z(\bar U_1)>0$.  Finally we note that $0< \nu(\bar U_1) = \eps\cdot \mu(U) \leq \eps$, and so  we deduce that for some $z\in Z$ the graphing $(\cal H,\nu_z)$ is an ergodic realization of $\Lambda$ such that $K_{\cal H,\nu_b} (\bar \pi) = K_{\cal G,\mu}(\pi)$ and furthermore  $0< \nu_z(\bar U_1)\leq \eps$.

Now we consider the following partition $\bar \rho$ of $(\cal H, \nu_z)$.  We let $\bar \rho_i := \bar \pi_i$ when 
$i\neq a,b$, and $\bar \rho_a := \pi_a \cup \bar U_1$, $\bar \rho_b = \bar \pi_b \setminus \bar U_1$. Since $\nu_z(\bar U_1) \leq \eps$, we have 
$d_\infty(W(\bar \rho), \al) \leq \eps$. 

On the other hand, it is easy to verify that we have 
$$ 
    K_{\cal H, \nu_z}(\bar \rho)  = K_{\cal H,\nu_b} (\bar \pi) - \nu_z(\partial_{\cal H} (\bar \pi_a,\bar U_1)). 
$$ 
Since each cluster defined by $\bar U_1$ contains a vertex connected to 
$\bar \pi_a$, and $\nu_z(\bar U_1) >0$, we deduce that  $\nu_z(\partial_{\cal H} (\bar \pi_a,\bar U_1))>0$, and therefore 
$$
    K_{\cal H, \nu_z}(\bar \rho)< K_{\cal H, \nu_b}(\bar \pi) = K_{\cal G,\mu} (\pi).
$$ 
This contradicts the Kazhdan-optimality of $\pi$, and finishes the proof.
\end{proof}

\section{Point processes on Property (T) groups}\label{sec:pp}

In this section we show that Property (T) of a lcsc group $G$ is characterized by Property (T) of the Palm equivalence relation associated to any of its free invariant point processes. Additionally, we extend the generalised Hutchcroft-Pete theorem to lcsc groups. We begin with a preliminary lemma. 

\begin{lemma}\label{lem:loc-fin}
    The (closed) Voronoi tessellation of a configuration is locally finite: if $\omega \in \MM$ and $h \in G$, then there exists an open neighbourhood of $h$ intersecting nontrivially only finitely many Voronoi cells of $\omega$.
\end{lemma}

\begin{proof}
    Let $R:=d(h,\omega)$, and note that by local finiteness of $\omega$ we have that $R > 0$. Moreover, there are only finitely many elements $g_1, g_2, \ldots, g_n$ of $\omega$ achieving the minimal distance.

    Again by local finiteness, there must exist some $\eps > 0$ such that $$
    B_{R+\eps}(h)\cap \omega = \{g_1,\dots,g_n\}.
    $$
    We claim that $B_{\eps/2} (h)$ intersects only the cells $V_{g_1}(\omega),\dots, V_{g_n}(\omega)$. 

    Observe that for $k \in B_{\eps/2} (h)$ we have that $$
    d(k,\omega) < \frac{\eps}{2} + R.
    $$
    However, we have that $$
    d(h,\omega\backslash \{g_1,\dots,g_n\}) \geq R + \eps.
    $$
    Therefore, we necessarily have that $$
    d (k,\omega\backslash \{g_1,\dots,g_n\}) \geq \frac{\eps}{2} + R.
    $$
 Thus $k$ can only belong to the Voronoi cells determined by $g_1,\dots,g_n$.
\end{proof}

\begin{theorem}\label{TForGroupImpliesTForPalm}
Let $G$ be a lcsc group and $\mu$ a (free) invariant point process on $G$. If $G$ has Property (T), then $(\mathbb M_0 , \mu_0)$ has Property (T).
\end{theorem}

\begin{proof}
Let $(\pi, \cal H)$ be a representation of $(\mathbb M_0 , \mu_0)$ with a sequence of almost invariant fields $(\xi_n)_n$ such that $\|\xi_n (\omega)\|_{\cal H_\omega} = 1$ a.s. By ergodicity we may assume that the field of Hilbert spaces $\cal H$ is constant. In a slight abuse of notation we let $\cal H$ denote this Hilbert space.

Our goal is to show that $(\pi,\cal H)$ admits an non-trivial invariant field. We achieve this by means of an auxiliary  representation $\rho$ of $G$ on the Hilbert subspace $\cal E$ of $L^2 (\mathbb M , \mu; \cal H)$ generated by the fundamental fields of $\cal H$. This representation is defined by \begin{equation*}
(\rho (g) f)(\omega) = \pi( \omega_0, \omega_{g}) f(g^{-1} \omega),
\end{equation*} 
for every $f \in  \cal E $ and $\omega \in \mathbb M$. Each $\rho (g)$ is unitary as it is the composition of a unitary with a measure-preserving shift. The map $\rho : G \rightarrow \cal U(L^2 (\cal M , \mu; \cal H))$ is also a homomorphism since, for every $f \in \cal E$ and $\omega \in \mathbb M$, we have that for any $g,h\in G$, \begin{align*}
(\rho(h)\rho(g) f) (\omega) &= \pi(\omega_0,\omega_{h})(\rho(g) f)(h^{-1} \omega) \\
&= \pi(\omega_0,\omega_{h}) \pi((h^{-1} \omega)_0, (h^{-1}\omega)_{g}) f((hg)^{-1}\omega)\\
&=\pi(\omega_0,\omega_{h}) \pi(\omega_h, \omega_{hg}) f((hg)^{-1}\omega)\\
&= (\rho (hg) f) (\omega).
\end{align*}
It is easy to see that for every $f\in \cal E$, the map $g\in G \mapsto \left< \rho(g) f, f\right>$ is measurable. Therefore $\rho$ is strongly continuous \cite[Lemma A.6.2]{bhv} and $(\rho,\cal E)$ is a representation of $G$.

To each field $\xi_n \colon \MM_0 \rightarrow \cal H$ we associate the function $\tilde\xi_n \in \cal E$ by letting $$
\tilde \xi_n (\omega) = \xi_n (\omega_0)
$$
for each $\omega \in \mathbb M$. The functions $\tilde \xi_n$ are indeed in $\cal E$. Moreover, by the Voronoi Inversion Formula we have that for every $n\in \NN$, \begin{equation*}
\|\tilde\xi_n\|^2 = \intensity (\mu) \int_{\mathbb M_0} \lambda (V_0 (\omega)) \|\xi_n (\omega)\|^2 d\mu_0 (\omega) = 1.
\end{equation*}

The following expression will be useful later. For every $n\in \NN$ and $\omega \in \MM$, we have that\begin{align*}
(\rho (g) \tilde\xi_n) (\omega) &= \pi( \omega_0, \omega_g) \xi_n ((g^{-1}\omega)_0 )\\
&= \pi( \omega_0, \omega_g) \xi_n (\omega_g) .
\end{align*}

We now show that the representation $(\rho,\cal E)$ almost has invariant vectors. To see this, fix a compact set $K \subseteq G$. For each $n \in \NN$, by strong continuity, we have that there exists $g_n \in K$ such that \begin{equation*}
\sup_{g\in K} \|\rho (g) \tilde\xi_n - \tilde\xi_n\| = \|\rho (g_n) \tilde\xi_n - \tilde\xi_n\|.
\end{equation*}
By compactness, upon possibly passing to a subsequence, we may assume that that $g_n \rightarrow g_0 \in K$, so 
\begin{equation*}
\sup_{g\in K} \|\rho (g) \tilde\xi_n - \tilde\xi_n\| \leq \|\rho(g_0^{-1}g_n)\tilde\xi_n - \tilde\xi_n\| + \|\rho(g_0) \tilde\xi_n - \tilde\xi_n\|.
\end{equation*}

We now proceed to show that both terms in the latter sum converge to zero. Let us begin the analysis with the second term. Observe that \begin{align*}
    \|\rho(g_0) \tilde\xi_n - \tilde\xi_n\|^2
    &= \intensity (\mu) \int_{\MM_0} \int_{V_0 (\omega)} \|\rho (g_0) \tilde\xi_n (h\omega) - \tilde\xi_n (h\omega)\|^2 dh d\mu_0 (\omega)\\
    &= \intensity (\mu) \int_{\MM_0} \int_{V_0 (\omega)}\|\pi((h\omega)_0,(h\omega)_{g_0})\tilde\xi_n (g_0^{-1} h\omega) - \xi_n ((h\omega)_0)\|^2 dh d\mu_0 (\omega)\\
    &=\intensity (\mu) \int_{\MM_0} \int_{V_0 (\omega)}\|\pi(\omega_{h^{-1}},\omega_{h^{-1}g_0})\xi_n (\omega_{h^{-1}g_0}) - \xi_n (\omega_{h^{-1}})\|^2 dh d\mu_0 (\omega).
\end{align*}
Hence, by almost invariance of the $\xi_n$ and the Dominated Convergence Theorem, we deduce that $\|\rho(g_0) \tilde \xi_n - \tilde\xi_n\| \rightarrow 0$.

It is then only left for us to show that if $(g_n)_n$ is a sequence in $G$ such that $g_n \rightarrow 0$, then $\|\rho(g_n) \tilde\xi_n - \tilde\xi_n\| \rightarrow 0$.  In order to prove this we start by applying the Voronoi Inversion Formula:
 \begin{align*}
 &\|\rho (g_n) \tilde\xi_n - \tilde\xi_n\|^2 \\
 &= \int_{\mathbb M}  \|\pi( \omega_0,  \omega_{g_n}) \xi_n (\omega_{g_n}) - \xi_n (\omega_0) \|^2 d\mu (\omega)\\
&= \intensity (\mu) \int_{\mathbb M_0} \int_{V_0 (\omega)} \|\pi( (h\omega)_0,  (h\omega)_{g_n}) \xi_n ((h\omega)_{g_n}) - \xi_n ((h\omega)_0) \|^2 dh d\mu_0 (\omega)\\
&= \intensity (\mu) \int_{\mathbb M_0} \int_{V_0 (\omega)} \|\pi( \omega_{h^{-1}},  \omega_{h^{-1}g_n}) \xi_n (\omega_{h^{-1}g_n}) - \xi_n (\omega_{h^{-1}}) \|^2 dh d\mu_0 (\omega)\\
\end{align*}
Note that for $\mu_0$-almost every $\omega \in \mathbb  M_0$, \begin{equation*}
\int_{V_0 (\omega)} \|\pi( \omega_{h^{-1}},  \omega_{h^{-1}g_n}) \xi_n (\omega_{h^{-1}g_n}) - \xi_n (\omega_{h^{-1}}) \|^2 dh \leq 4 \lambda (V_0 (\omega)),
\end{equation*}
and $\int_{\MM_0}\lambda (V_0 (\omega))]d\mu_0 (\omega) = \intensity(\mu)^{-1}$, so by the dominated Convergence Theorem it suffices to show that for $\mu_0$-almost every $\omega \in \mathbb  M_0$, 
\begin{equation*}
\int_{V_0 (\omega)} \|\pi( \omega_{h^{-1}},  \omega_{h^{-1}g_n}) \xi_n (\omega_{h^{-1}g_n}) - \xi_n (\omega_{h^{-1}}) \|^2 dh.
\end{equation*}
Again by the Dominated Convergence Theorem, it suffices to show that for $\mu_0$-almost every $\omega\in \MM_0$ and for Haar-almost every $h \in V_0 (\omega)$, $$
\|\pi( \omega_{h^{-1}},  \omega_{h^{-1}g_n}) \xi_n (\omega_{h^{-1}g_n}) - \xi_n (\omega_{h^{-1}}) \|\to 0.
$$

By Lemma \ref{lem:loc-fin}, for each $h \in V_0 (\omega)$ there exists an open neighbourhood of the identity $U_h$ such that $h^{-1} U_h$ intersects only finitely many other Voronoi cells. These cells are determined by a finite subset $F_h \subseteq \omega$. We then have that for $n$ sufficiently large, $$
 \|\pi( \omega_{h^{-1}},  \omega_{h^{-1}g_n}) \xi_n (\omega_{h^{-1}g_n}) - \xi_n (\omega_{h^{-1}}) \| \leq \sum_{g\in F_h} \|\pi( \omega_{h^{-1}},  \omega_{g}) \xi_n (\omega_{g}) - \xi_n (\omega_{h^{-1}}) \|,
$$
and the latter converge to $0$ on a $\mu_0$-conull re-rooting invariant set. We then conclude that $(\tilde\xi_n)_n$ is an almost invariant sequence of vectors for the representation $\rho$.

Since $G$ has Property (T), then there exists an invariant function $\tilde\eta \in  \cal E$. By construction of $\rho$, we must have that for $\mu_0$-almost every $\omega \in \mathbb M_0$, the function $h\in V_0 (\omega) \mapsto \tilde\eta (h \omega)$ is essentially constant, taking a value which we denote by $\eta (\omega)$. Now it is easy to see that the field $\omega \in \mathbb M_0 \mapsto \eta (\omega)$ is invariant for $(\pi,\cal H)$. It follows that $(\mathbb M_0 , \mu_0)$ has Property (T) by definition.
\end{proof}

Now we prove the converse of the above Theorem.

\begin{theorem}\label{otherway}
Let $G$ be a lcsc group and $\mu$ a free invariant point process. If $(\mathbb M_0,\mu_0)$ has Property (T), then $G$ has Property (T).
\end{theorem}

\begin{proof}
For convenience, in this proof we will use the groupoid description of the Palm equivalence relation, so for $\omega \in \MM_0$ and $g\in \omega$ we will let $(g,\omega)$ denote the edge $(g^{-1}\omega,\omega)$ of the re-rooting equivalence relation.

Let $(\pi, H)$ be a unitary representation of $G$ with an almost invariant sequence of normalised vectors $(\xi_n)_n$. Consider the representation  $(\tilde\pi, \cal H)$ of $\MM_0$ defined by  the measurable field of Hilbert spaces $\omega \in \mathbb M_0 \mapsto \cal H_\omega := H $ and $\tilde\pi (g,\omega) : \cal H_\omega \rightarrow \cal H_{g^{-1}\omega}$ defined by $\tilde\pi (g, \omega) = \pi(g)$. It is easy to see that the sequence of constant fields $(\omega \in \mathbb M_0 \mapsto \xi_n)_{n\in \NN}$ is almost invariant for $(\tilde\pi, \cal H)$, so by Pichot's Lemma there is a sequence of invariant fields $\eta_n$ such that $\|\eta_n (\omega)\| = 1$ and $\|\eta_n (\omega) - \xi_n \| \rightarrow 0$ a.s. 

We claim that the vectors
\begin{equation*}
\rho_n = \intensity (\mu)\int_{\mathbb M_0} \int_{V_0 (\omega)} \pi(h)\eta_n (\omega)dh d\mu_0 (\omega) \in H
\end{equation*} 
are $(\pi, H)$-invariant and non-null for sufficiently large $n$.

Firstly, we show that they are non-null for sufficiently large $n$. For $\pi$ being unitary, \begin{align*}
\frac{1}{\intensity (\mu)}\|\rho_n - \xi_n\| &\leq \int_{\mathbb M_0} \int_{V_0 (\omega)} \|\pi(h)\eta_n (\omega) - \xi_n\|dh d\mu_0 (\omega) \\
&\leq \int_{\mathbb M_0} \lambda (V_0 (\omega))\|\eta_n (\omega) - \xi_n\| d\mu_0 (\omega) \\
&+ \int_{\mathbb M_0} \int_{V_0 (\omega)} \|\pi(h)\xi_n  - \xi_n\| dh d\mu_0 (\omega),
\end{align*}
where $\lambda$ is the Haar measure on $G$. Let us study the two terms separately.

The integrand of the first term may be bounded as \begin{equation*}
\lambda (V_0 (\omega))\|\eta_n (\omega) - \xi_n\|  \leq 2 \lambda (V_0(\omega)),
\end{equation*}
and $\int_{\mathbb M_0}\lambda (V_0(\omega))d \mu_0 (\omega) = \intensity(\mu)^{-1} < \infty$, where $\gamma$ is the intensity of the process. Hence, the Dominated Convergence Theorem applies and, as $\|\eta_n (\omega) - \xi_n\| \rightarrow 0$ a.s., then
 \begin{equation*}
\int_{\MM_0} \lambda(V_0 (\omega)) \|\eta_n (\omega) - \xi_n\| d\mu_0 (\omega)\to 0.
\end{equation*}

Now, regarding the second term, let $(K_m)_m$ be an increasing sequence of compact subsets of $G$ such that $G = \bigcup_m K_m$. Then \begin{align*}
\int_{\mathbb M_0} \int_{V_0 (\omega)} &\|\pi(h)\xi_n  - \xi_n\| dh d\mu_0 (\omega) \\
&= \int_{\mathbb M_0} \left(\int_{V_0 (\omega)\cap K_m} \|\pi(h)\xi_n  - \xi_n\| dh + \int_{V_0 (\omega)\backslash K_m} \|\pi(h)\xi_n  - \xi_n\| dh \right) d\mu_0 (\omega) \\
& \leq \left(\sup_{g \in K_m}  \|\pi(g)\xi_n  - \xi_n\|\right) \E_{\mu_0} [\lambda (V_0 (\omega) \cap K_m)] + 2 \E_{\mu_0} [\lambda (C_0 (\omega)\backslash K_m)] \\
&\leq \left(\sup_{g \in K_m}  \|\pi(g)\xi_n  - \xi_n\|\right)\intensity(\mu)^{-1} + 2 \E_{\mu_0} [\lambda (V_0 (\omega)\backslash K_m)].
\end{align*}
By the Dominated Convergence Theorem, \begin{equation*}
\lim_{m \rightarrow \infty} \E_{\mu_0} [\lambda (V_0 (\omega)\backslash K_m)] =  \E_{\mu_0} [\lim_{m \rightarrow \infty} \lambda (V_0 (\omega)\backslash K_m)] = 0,
\end{equation*} 
so $ 2 \E_{\mu_0} [\lambda (V_0 (\omega)\backslash K_M)] \to 0$. By almost invariance of the sequence $\xi_n$,  we also have that $$
\left(\sup_{g \in K_M}  \|\pi(g)\xi_n  - \xi_n\|\right) \rightarrow 0,$$
so again $\left(\sup_{g \in K_M}  \|\pi(g)\xi_n  - \xi_n\|\right)\gamma^{-1} \to 0$.

The above implies that for $n$ sufficiently large, $\|\rho_n - \xi_n \| < 1$, so $\rho_n \neq 0$.

We conclude by showing that the $\rho_n$ are all $\pi$-invariant. By the Voronoi Inversion Formula,
\begin{equation*}
\rho_n = \int_{\mathbb M}  \pi (X_0 (\omega)) \eta_n (\omega_0) d\mu (\omega) .
\end{equation*}
Therefore, for any $g\in G$,
\begin{align*}
\pi (g)\rho_n &=  \int_{\mathbb M}  \pi (gX_0 (\omega)) \eta_n (\omega_0) d\mu (\omega) && \\
&= \int_{\mathbb M}\pi (g X_0(g^{-1} \omega)) \eta_n ( (g^{-1}\omega)_0) d \mu (\omega) && (\mu\text{-invariance})\\
&= \int_{\mathbb M}\pi ( X_g( \omega)) \eta_n (X_0 (g^{-1}\omega)^{-1}g^{-1}\omega) d \mu (\omega) && (X_g (\omega) = g X_0 (g^{-1} \omega))\\
&= \int_{\mathbb M}\pi ( X_g( \omega)) \eta_n (X_g( \omega)^{-1} X_0 (\omega) X_0(\omega)^{-1} \omega)) d \mu (\omega) && \\
&= \int_{\mathbb M}\pi ( X_g( \omega)) \pi (X_g( \omega)^{-1} X_0 (\omega) )\eta_n (X_0(\omega)^{-1} \omega)) d \mu (\omega) && \\
&= \int_{\mathbb M}\pi ( X_0( \omega)) \eta_n ( X_0(\omega)^{-1} \omega)) d \mu (\omega) && (\eta_n\ \text{invariant})\\
&= \rho_n. &&
\end{align*}
Therefore, $(\pi, H)$ has invariant vectors. It follows that $G$ has Property (T) by definition.
\end{proof}

\begin{theorem}\label{thm:realgs}
    Let $G$ be a noncompact lcsc group with Property (T). Then $G$ has cost one.
\end{theorem}

\begin{proof}
    Let $G \acts (X, \mu)$ be an essentially free action of $G$, and $Y \subseteq X$ a cross-section. Then the associated cross-section equivalence relation has Property (T) by Theorem \ref{TForGroupImpliesTForPalm}. Therefore it admits a cost one extension by Theorem \ref{thm:HP}. As discussed in Section \ref{PPprelims}, this extension of the cross-section equivalence relation arises as the cross-section of some extension of $G \acts (X,\mu)$, and therefore $G$ has cost one.
\end{proof}

\section*{Appendix. The proof for countable groups}

In this appendix we give a streamlined, self-contained proof of the main result of the paper (Theorem \ref{thm:HP.intro}) for countable groups. 

\subsection{Cost, probabilistically}

\begin{definition}
    The \emph{configuration space} of $\Gamma$ is $\MM = \{0,1\}^\Gamma$ with the usual left shift action. We refer to elements $\omega \in \MM$ as \emph{configurations}.

    For $d \in \NN$, we denote by $[d]^\Gamma$ the space of \emph{$d$-colourings of $\Gamma$}. Here $[d] = \{1, 2, \ldots, d\}$, and we think of its elements $c \in [d]$ as coming from a finite set of colours.
\end{definition}

We equip both spaces $\MM$ and $[d]^\Gamma$ with the product topology. 

Configurations can be thought of as subsets of $\Gamma$; an element of $\Gamma$ belongs to the subset if and only if it is assigned 1 by the configuration. A $d$-colouring of $\Gamma$ is simply an assignment of a colour to each element. In particular, there are no graph theoretic restrictions on which colours are assigned to each vertex.

Recall that a \emph{random element} $X$ of a measurable space $\Xi$ is a measurable function $X \colon \Omega\to \Xi$, where $(\Omega, \PP)$ is a standard probability space. The \emph{distribution} of $X$ is then the pushforward measure $X_*(\PP)$.

\begin{definition}
    A \emph{random subset} of $\Gamma$ is a random element $\Pi$ of $\MM$. It is \emph{invariant} if $\gamma \Pi$ and $\Pi$ have the same distribution for all $\gamma \in \Gamma$.

    Similarly, an \emph{invariant random $d$-colouring} of $\Gamma$ is a random element $\Upsilon$ of $[d]^\Gamma$ such that $\gamma \Upsilon$ has the same distribution as $\Upsilon$ for all $\gamma\in \Gamma$.  
\end{definition}

If $\Upsilon$ is an invariant random $d$-colouring of $\Gamma$, each colour $c \in [d]$ defines an invariant random subset $\Upsilon_c$ consisting of all points coloured $c$ in $\Upsilon$. Explicitly,
    \[
        \Upsilon_c := \{ \gamma \in \Gamma \mid \Upsilon(\gamma) = c \} = \Upsilon^{-1}(c).
    \]

\begin{definition}
    The \emph{intensity} of an invariant random subset $\Pi$ is
    \[ 
        \intensity(\Pi) = \PP[o \in \Pi],
    \]
    where $o \in \Gamma$ denotes the identity.
\end{definition}

\begin{remark}
    By invariance of $\Pi$, if $\Gamma$ is infinite and $\intensity(\Pi) > 0$, then $\Pi$ has infinitely many points almost surely.
\end{remark}

Informally speaking, the \emph{cost} of an invariant random object is the ``cheapest way'' to ``connect it up''. We do this by equipping its points with a connected graph structure.

\begin{definition}
    A \emph{factor graph (of $\Pi$)} is a measurable and equivariant map $\mathscr{G} : \MM \to \{0,1\}^{\Gamma \times \Gamma}$ such that $\mathscr{G}(\Pi) \subseteq \Pi \times \Pi$ almost surely. We view $\mathscr{G}(\Pi)$ as defining a (directed) graph structure on $\Pi$. 

    The \emph{cost} of an invariant random subset $\Pi$ is defined by
    \[
        \cost(\Pi) -1 = \frac{1}{2} \inf_{\mathscr{G} }\EE[\deg_o(\mathscr{G}(\Pi))] - \intensity(\Pi),
    \]
    where the infimum ranges over all connected factor graphs of $\Pi$.
\end{definition}

\begin{example}
     If $\Pi = \Gamma$ almost surely, then its cost is the \emph{rank} of $\Gamma$, that is, the minimum size of a generating set.
\end{example}

\begin{remark}
    The above notion of cost connects with the usual notion of cost, as in \cite{Gab}, as follows. Introduce the subset of \emph{rooted configurations}
    \[
        \MMo = \{ \omega \in \MM \mid o \in \omega \},
    \]
    and note that $\MMo$ is a complete section for the action $\Gamma \acts \MM \setminus \{\empt\}$. The above definition is Gaboriau's induction formula for this section.
\end{remark}

Now let $\Gamma = \langle S \rangle$ be finitely generated by a symmetric subset $S \subset \Gamma$. 

\begin{definition}
    If $\Pi$ is an invariant random subset, a \emph{cluster of $\Pi$} is the set of vertices of a connected component of the subgraph of $\Cay(\Gamma, S)$ induced by $\Pi$.
\end{definition}

\begin{remark}
    If $\Pi$ has positive intensity and the subgraph of $\Cay(\Gamma,S)$ it induces has finitely many components, then $\Pi$ must have \emph{at least} one cluster.
\end{remark}

The next lemma is essentially immediate from the definitions:

\begin{lemma}
    If $\Pi \subset \Gamma$ is an invariant random subset with intensity at most $\iota$ and a unique cluster, then
    \[
        \cost(\Pi) \leq 1 + \iota \abs{S}.
    \]
\end{lemma}

\begin{corollary}\label{cor:fin.clusters}
    If $\Gamma$ admits a nonempty invariant random subset $\Pi$ with intensity at most $\iota$ and \emph{finitely many} clusters, then there exists an invariant random subset $\widetilde{\Pi}$ of $\Gamma$ with intensity at most $\iota$ and 
    \[
        \cost(\widetilde{\Pi}) \leq 1 + \iota \abs{S}.
    \]
\end{corollary}

The corollary follows from the lemma in two different ways:

\begin{enumerate}
    \item One can use a vanishing sequence of markers to directly show that $\widetilde{\Pi} = \Pi$ satisfies the inequality, or
    \item One can let $\widetilde{\Pi}$ be one of the clusters of $\Pi$ chosen uniformly at random, and observe that $\intensity(\widetilde{\Pi}) < \intensity(\Pi)$.
\end{enumerate}

The first method has the advantage of not requiring any additional randomness (that is, passing to an extension of the implicit probability space). However, as our goal is to simply produce a cost one action, we are free to pass to whichever extension we please.

Recall that if $\Gamma \acts (X_i, \mu_i)$ are pmp actions, then the cost of their product is at most the cost of any of the individual factors. Symbolically:
\[
    \cost\left(\Gamma \acts \left(\prod_i X_i, \bigotimes_i \mu_i\right)\right) \leq \cost(\Gamma \acts (X_i, \mu_i)),
\]
This follows as cost is nondecreasing under factors (here we take the coordinate projection). This observation and the above lemma implies:

\begin{proposition}\label{StrategyIdea}
    Suppose $\Gamma$ admits nonempty invariant random subsets with finitely many clusters and arbitrarily small intensity. Then $\Gamma$ admits a cost one action.
\end{proposition}

\subsection{Property (T) and weak convergence of colourings}

In this appendix, we will use the following characterization of Property (T) as a definition.

\begin{theorem}[Glasner-Weiss \cite{GW}]\label{GlasnerWeiss}
    A group has Property (T) if and only if for every $d \geq 2$ the set of \emph{ergodic} invariant $d$-colourings is closed (in the weak topology).
\end{theorem}

A Borel subset $U \subseteq [d]^\Gamma$ is \emph{invariant} if $\gamma U = U$ for every $\gamma \in \Gamma$. Recall that an invariant random $d$-colouring $\Upsilon$ is \emph{ergodic} if $\PP[\Upsilon \in U] = 0$ or $1$ for all invariant subsets $U \subseteq [d]^\Gamma$.

The space of probability measures on $[d]^\Gamma$, denoted by $\Prob ([d]^\Gamma)$ is a compact space when endowed with the \emph{weak topology}. Recall that a sequence of probability measures $(\mu_n)_{n\in \mathbb N}$ weakly converges to $\mu$ if for any $F \subset \Gamma$ finite and $\alpha \in [d]^F$ we have
\[
\mu_n(\{\omega \in [d]^\Gamma : \omega|_F = \alpha\}) \to \mu (\{\omega \in [d]^\Gamma : \omega|_F = \alpha\}),
\]
that is, if the statistical behaviour of the random colourings in any finite window tends to that of the limiting random colouring.
A sequence of invariant random $d$-colourings $(\Upsilon_n)_{n\in \mathbb N}$ converges to $\Upsilon$ if the corresponding sequences of distributions weakly converge.

\subsection{Invariant random equivalence relations}

Invariant random equivalence relations (IREs) on groups were introduced in \cite{TDthesis} and later studied in \cite{KechrisSpaceOfEqRels} and, more recently, in \cite{IREs}.

For a countable group $\Gamma$, we denote by $\cal E_\Gamma \subseteq 2^\Gamma$ the set of equivalence relations on $\Gamma$ endowed with the subspace topology. The group $\Gamma$ acts on $\cal E_\Gamma$ by the left shift, i.e., $\gamma r (x,y) = r (\gamma^{-1}x, \gamma^{-1}y)$
for $\gamma, x,y \in \Gamma$. We write $x r y$ if and only if $r(x,y) = 1$.

\begin{definition}
    An \textbf{invariant random equivalence relation (IRE)} is a random element $R$ of $\mathcal E_\Gamma$ with invariant distribution.
    \end{definition}

We now define the examples of IREs that we will make use of.

\begin{example}[Cluster equivalence relation]
    Let $\Pi$ be an invariant random subset of $\Gamma$. For each $x, y \in \Gamma$ we may define $x R_\Pi y$ if and only if $x$ and $y$ belong to the same $\Pi$-cluster.
\end{example}

A \emph{rerooting invariant event} is a Borel subset $U \subseteq \mathcal E_\Gamma$ such that for every $r \in U$ and $\gamma \in [o]_r$ we have $\gamma r \in U$. An IRE $R$ is \emph{ergodic} if $\PP [R \in U] = 0$ or $1$ for every rerooting invariant event.

\begin{example}[Colouring of an IRE]
    An \emph{ergodic colouring} of an ergodic IRE $R$ is a coupling $(R, \Upsilon)$ where $\Upsilon$ is an ergodic invariant random colouring with the following property. For every $x,y \in \Gamma$, if $xRy$, then $\Upsilon (x) = \Upsilon (y)$ almost surely. That is, the colours of $\Upsilon$ are constant on $R$-classes. 

    The explicit example we will make use of is the following: let $R$ be an IRE and $\eps >0$. The \emph{Bernoulli colouring of $R$}, denoted by $\Upsilon_R$ is the invariant random subset defined as follows. First, independently at random assigning to each $R$-equivalence class $C$ the value 1 with probability $\eps$ and $0$ otherwise. Then, give each element $x \in C$ the value assigned to $C$. The construction of this object is discussed in Section \ref{clust.ext}.
\end{example}

The following proposition follows from \cite[Corollary 5.2.11]{hector}. It can also be bypassed (as in the more general proof in the paper) by considering the ergodic decomposition, but it is of independent interest.

\begin{proposition}\label{prop:ergodic.ber}
    Let $R$ be an ergodic IRE. Suppose that for every ergodic colouring $(R,\Upsilon)$ there are infinitely many $R$-classes in $\Upsilon$ almost surely. Then the Bernoulli colouring $\Upsilon_R$ is ergodic for every $0 < \eps <1$.
\end{proposition}

\subsection{Kazhdan optimal colourings}\label{NewResultsSection}

\begin{definition}
    Suppose $\Upsilon$ is an invariant random $d$-colouring and $\delta \geq 0$. We say that $\Upsilon$ is \emph{$\delta$-balanced} if for all colours $c \in [d]$
    \[
        \abs{\intensity(\Upsilon_c) - 1/d} \leq \delta,
    \]
    that is, if $\PP[o \text{ is coloured } c \text{ in } \Upsilon] \approx_\delta 1/d$.
\end{definition}

\begin{example}[Bernoulli colouring]
    Construct an invariant random $d$-colouring $\Upsilon$ by assigning colours to each point of $\Gamma$ independently and uniformly at random. This is the \emph{Bernoulli} colouring, and it is $0$-balanced. Sampling the colours from a distribution other than the uniform, one can construct further examples of $\delta$-balanced invariant random colourings with $\delta$ strictly bigger than zero.
\end{example}

\begin{example}[Constant colouring]
    Construct an invariant random $d$-colouring $\Upsilon$ by choosing \emph{one} colour uniformly at random and colouring \emph{all} $\gamma \in \Gamma$ by that colour. The outcome is again 0-balanced, and clearly different from the above colouring. Note that this colouring is nonergodic (we will make use of this observation).
\end{example}

\emph{Fix a symmetric finite generating set $S \subset \Gamma$}. We introduce a notion of expansion depending on this choice of generators.

\begin{definition}
    Let $\Upsilon$ be an invariant random colouring of $\Gamma$. Its \emph{$S$-expansion} is
    \[
        \mathcal{E}_S(\Upsilon) = \EE\abs{ \{ s \in S \mid \Upsilon(s) \neq \Upsilon(o) \} },
    \]
    that is, the expected number of neighbours of the root $o \in \Gamma$ assigned a colour \emph{differing} from $o$'s.

    The \emph{$(d, S,\delta)$-Kazhdan constant} of $\Gamma$ is
    \[
        K(d, S, \delta) = \inf_\Upsilon \mathcal{E}_S(\Upsilon),
    \]
    where the infimum ranges over all $\delta$-balanced \emph{ergodic} invariant random $d$-colourings $\Upsilon$.
\end{definition}

Note that $\mathcal{E}_S(\Upsilon) = 0$ if and only if $\Upsilon$ concentrates on the constant colourings. 

\begin{proposition}
    Suppose that $\Gamma$ has Property (T) and $\delta < 1/d$. Then:
    \begin{enumerate}
        \item There exists an ergodic invariant random $d$-colouring $\Upsilon$ attaining the infimum of the Kazhdan constant, that is $\mathcal{E}_S(\Upsilon) = K(d,S,\delta)$).
        \item Additionally, $K(d, S, \delta) > 0$.
    \end{enumerate}
\end{proposition}

\begin{proof}
    \noindent\begin{enumerate}
        \item Choose $\Upsilon_n$ to be ergodic invariant random $d$-colourings with $\mathcal{E}_S(\Upsilon_n)$ decreasing to $K(d, S, \delta)$. By compactness we may assume that $\Upsilon_n$ weakly converges to some invariant random $d$-colouring $\Upsilon$. By the Glasner-Weiss theorem, $\Upsilon$ must be ergodic. Since the function $[d]^\Gamma \to \ZZ$ given by $\omega \mapsto \abs{ \{ s \in S \mid \omega(s) \neq \omega(o) \} }$ is continuous, weak convergence implies that $\mathcal{E}_S(\Pi) = K(d,S,\delta)$. 
        \item Suppose $K(d, S, \delta) = 0$. Then there exists an ergodic invariant random $d$-colouring $\Upsilon$ with $\mathcal{E}_S(\Upsilon) = 0$. But, as observed in the preamble of the proposition, such $\Upsilon$ must concentrate on the constant colourings, contradicting ergodicity of $\Upsilon$, as $\delta < 1/d$.\qedhere
    \end{enumerate}
\end{proof}

\begin{definition}
    An ergodic invariant random $d$-colouring $\Upsilon$ of $\Gamma$ with $\mathcal{E}_S(\Upsilon) = K(d,S,\delta)$ is deemed \emph{Kazhdan optimal}.
\end{definition}

\begin{theorem}\label{MainTool}
    Suppose $\Gamma$ is a countably infinite group with Property (T). Then $\Gamma$ has the \emph{sparse finitely many clusters property}: for every $\iota > 0$, there exists an invariant random subset $\Pi$ of $\Gamma$ with intensity at most $\iota$ and a unique cluster. In particular $\Gamma$ has cost 1.
\end{theorem}

\begin{proof}
    The sparse unique infinite cluster property implies cost 1 by Proposition \ref{StrategyIdea}. We will show that, in a Kazhdan optimal coloring, a nonempty subset of some colour class has finitely many clusters. First, choose $d$ and $\delta$ such that $1/d + \delta < \iota$, and let $\Upsilon$ be a $(d,S, \delta)$-Kazhdan optimal colouring. Note that each colour class of $\Upsilon$ has intensity at most $\iota$ by definition. 

     If some colour $c \in [d]$ has the property that $\Upsilon_c$ has finitely many infinite clusters, then we are done. Therefore, we may \emph{assume that every colour class has infinitely many clusters}. Since $K(d,S,\delta)>0$, there exist two colours $r, b \in [d]$ with the property that
    \[
        \PP[\exists s \in S, o \text{ is coloured } r \in \Upsilon, s \text{ is coloured } b \in \Upsilon] > 0.
    \]
    That is, when sampling from $\Upsilon$, a red root has a blue neighbour with positive probability. It follows from invariance and ergodicity that almost surely there are infinitely many edges $(\gamma, \gamma s)$ in the $\Upsilon$-coloured Cayley graph of $\Gamma$ such that $\gamma$ is red and $\gamma s$ blue. Without loss of generality, we can assume that $\intensity(\Upsilon_r) \geq 1/d$ and $\intensity(\Upsilon_b) \leq 1/d$. 

    Suppose one of the colours (red, say) has finitely many clusters at distance one from the blue components. Then we can measurably select those finitely many clusters, concluding the proof. So \emph{suppose there are infinitely many red and blue clusters at distance one from each other.} We will show that this contradicts Kazhdan optimality. 
    
    Let $\Upsilon_r$ be the $\delta$-Bernoulli colouring of the equivalence relation given by belonging to the same red cluster at distance one from a blue cluster. We may assume without loss of generality that $\Upsilon_r$ is ergodic, for otherwise, by Proposition \ref{prop:ergodic.ber}, there is a coupling $(\Upsilon,\Pi)$ such that $\Pi$ selects finitely many red $\Upsilon$ clusters, concluding the proof. Let $\widetilde{\Upsilon}$ be obtained by assigning colour blue to those red clusters with $\Pi_r$ colour 1, and otherwise keeping the $\Upsilon$ colours. 
    
    Note that $\mathcal E_S (\widetilde{\Upsilon}) < \mathcal E_S (\Upsilon)$. Indeed, by ergodicity, in almost every component there are edges which formerly had end-vertices red-blue, but now have end-vertices blue-blue. Edges with distinct colours for their end-vertices in $\tilde\Upsilon$ already had distinct colour end-vertices in $\Upsilon$, so claim is proven. 

    Since $\intensity (\Upsilon_b) \leq 1/d \leq \intensity (\Upsilon_r)$, we have that $1/d - \delta \leq \intensity (\tilde\Upsilon_r) \leq \intensity (\tilde\Upsilon_b) \leq 1/d + \delta$. Hence we have produced a $\delta$-balanced ergodic invariant random $d$-colouring $\tilde\Upsilon$ with $S$-expansion strictly smaller than the Kazhdan constant; a contradiction.
\end{proof}

\def\MR#1{}
\bibliographystyle{amsalpha} 
\bibliography{bib}

\end{document}